\documentclass[a4paper,final]{article}

\usepackage{amsthm}
\newtheorem{teo}{Theorem}
\newtheorem{prop}[teo]{Proposition}
\newtheorem{lem}[teo]{Lemma}
\newtheorem{cor}[teo]{Corollary}
\newtheorem{question}[teo]{Question}

\theoremstyle{definition}
\newtheorem{rem}[teo]{Remark}

\usepackage{amssymb}
\usepackage{amsmath}

\usepackage[dvips]{graphicx}
\usepackage{psfrag}

\newcommand{\calC}{{\cal C}}
\newcommand{\calT}{{\cal T}}
\newcommand{\calF}{{\cal F}}
\newcommand{\calR}{{\cal R}}
\newcommand{\calP}{{\cal P}}

\renewcommand{\int}{{\rm int}}

\newcommand{\vigone}{\mbox{${\rm Vig}_1$}}
\newcommand{\vigtwo}{\mbox{${\rm Vig}_2$}}
\newcommand{\vigthree}{\mbox{${\rm Vig}_3$}}
\newcommand{\vigi}{\mbox{${\rm Vig}_*$}}

\psfrag{V1}{\small $\vigone$}
\psfrag{V2}{\small $\vigtwo$}
\psfrag{V3}{\small $\vigthree$}

\newcommand{\tone}{\mbox{${\rm T}_1$}}
\newcommand{\ttwo}{\mbox{${\rm T}_2$}}
\newcommand{\tthree}{\mbox{${\rm T}_3$}}
\newcommand{\tfour}{\mbox{${\rm T}_4$}}
\newcommand{\ti}{\mbox{${\rm T}_*$}}
\newcommand{\bubble}{\mbox{${\rm B}$}}
\newcommand{\saddle}{\mbox{${\rm S}$}}
\psfrag{T1}{\small $\tone$}
\psfrag{T2}{\small $\ttwo$}
\psfrag{T3}{\small $\tthree$}
\psfrag{T4}{\small $\tfour$}
\psfrag{B}{\small $\bubble$}
\psfrag{S}{\small $\saddle$}

\newcommand{\wall}{\mbox{${\rm W}$}}
\psfrag{W}{\small $\wall$}

\newcommand{\onetotwo}{\mbox{$1{\rm to}2$}}
\newcommand{\zerototwo}{\mbox{$0{\rm to}2$}}
\psfrag{1to2}{\small $\onetotwo$}
\psfrag{0to2}{\small $\zerototwo$}

\newcommand{\field}{\mathbb F}

\newcommand{\twelvejsym}[4]{
\small{
\left|
\begin{array}{@{}c@{\ \,}c@{\ \,}c@{}}
#1 \\
#2 \\
#3 \\
#4
\end{array}
\right|
}
}

\newcommand{\ptwoirred}{$\mathbb{P}^2$--irreducible}
\newcommand{\RPthree}{\mathbb{RP}^3}

\newcommand{\doubleidx}[2]{\tiny{\begin{array}{c} #1 \\ #2 \end{array}}}

\psfrag{SPP}{\small ${\rm SPP}$}
\psfrag{p}{\small $p$}
\psfrag{pp}{\small $p'$}
\psfrag{D}{\small $D$}
\psfrag{Dp}{\small $D'$}
\psfrag{R}{\small $R$}
\psfrag{gamma}{\small $\gamma$}

\newcommand{\co}{\colon\thinspace}

\title{Invariants of Closed 3--Manifolds\\via Nullhomotopic Filling Dehn Spheres}

\author{Gennaro {\sc Amendola}}

\begin{document}

\maketitle

\begin{abstract}
  We provide a calculus for the presentation of closed 3--manifolds via
  nullhomotopic filling Dehn spheres and we use it to define an
  invariant of closed 3--manifolds by applying the state-sum machinery.
  As a potential application of this invariant, we show how to get lower bounds
  for the Matveev complexity of \ptwoirred\ closed 3--manifolds.
  We also describe an easy algorithm for constructing a nullhomotopic filling Dehn
  sphere of each closed 3--manifold from any of its one-vertex triangulations.
\end{abstract}

\vspace{1pt}

\begin{center}
\begin{small}
{\bf Keywords}\\[4pt]
3--manifold, immersed surface, invariant, state sum, complexity.

\vspace{.5cm}

{\bf MSC (2000)}\\[4pt]
57M27 (primary), 57R42 (secondary).
\end{small}
\end{center}

\section*{Introduction}

A {\em presentation} of a class of topological objects
(in our case {\em closed $3$--manifolds}) is a class of
combinatorial objects (in our case {\em nullhomotopic filling Dehn spheres}),
such that each combinatorial object defines
(say ``presents'') a unique topological object and each topological
object is presented by at least one combinatorial object.
A ({\em finite}) {\em calculus} for a presentation is a (finite) set of
moves on the combinatorial objects, such that two combinatorial
objects present the same topological object if and only if they are
related to each other by a finite sequence of moves in the given set.

Presentations and calculuses are fundamental tools for studying
3--manifolds and for constructing invariants, in fact
they translate a topological problem into a combinatorial and perhaps
simpler one.
For instance, an invariant on the class of topological objects can be
defined on the class of combinatorial objects, checking that it is preserved by
the moves of the calculus.

For closed 3--manifolds, there are several different types of presentations,
{\em e.g.}~triangulations, Heegaard diagrams, surgery (on
links) and spines.
In the present work we concentrate on {\em nullhomotopic filling Dehn spheres}, which
dually can be thought as a particular class of {\em cubulations} (see, for instance, Aitchison and Matsumotoi and Rubinstein~\cite{Aitchison-Matsumotoi-Rubinstein}, Funar~\cite{Funar}, Babson and Chan~\cite{Babson-Chan}).
The fact that nullhomotopic filling Dehn spheres present closed 3--manifolds
is already known (see, for instance, Montesinos-Amilibia~\cite{Montesinos} and Vigara~\cite{Vigara:present}).
We will provide here a proof of this result by using a very simple and efficient construction.
Such a construction is already known and studied (see, for instance, Shtan$'$ko and Shtogrin~\cite{Shtanko-Shtogrin}, Dolbilin and Shtan$'$ko and Shtogrin~\cite{Dolbilin-Shtanko-Shtogrin} and Funar~\cite{Funar}), but we have not found any application to nullhomotopic filling Dehn spheres in literature.

We will also provide a finite calculus for this presentation, deducing it
from another one, described by Vigara~\cite{Vigara:calculus}, which has been derived from the more general Homma--Nagase calculus~\cite{Homma-NagaseI, Homma-NagaseII} (see also Hass and Hughes~\cite{Hass-Hughes} and Roseman~\cite{Roseman}).
The main feature of our calculus consists in being local ({\em i.e.}~in order to apply a move, it is enough to look only at the portion of the nullhomotopic filling Dehn sphere involved in the move).
On the contrary, Vigara's calculus is very interesting and natural, but it has the drawback of not being local; hence, it is not useful for applying the state-sum machinery to define an invariant analogous to the Turaev--Viro one~\cite{Turaev-Viro}.
More precisely, Turaev and Viro used the Matveev--Piergallini calculus for
spines~\cite{Matveev:calculus, Piergallini:calculus} to define an invariant
for closed 3--manifolds as follows.
They defined a state sum for each spine ({\em i.e.}~a polynomial whose summands correspond to
different ``colourings'' of the spine) and they proved that, if its variables satisfy some equations ({\em e.g.}~the so-called Biedenharn--Elliott equations~\cite{Biedenharn-Louck, Turaev-Viro, Turaev}), the state sum is an invariant of the closed 3--manifold presented by the spine.
The equations come from the moves of the Matveev--Piergallini calculus and describe how the polynomial changes when the moves are applied.

We will use a framework analogous to Turaev and Viro's one.
Namely, we will first define the state sum for a nullhomotopic filling Dehn sphere.
Afterwards, we will study how it changes when a move of our calculus is applied, and we will prove that the difference between the state sums of two nullhomotopic filling Dehn spheres of the same closed 3--manifold is an element of a particular ideal of the polynomial ring.
(It is at this point that we will use the fact that our calculus is local, because in such a case the alteration due to the moves can be understood and computed explicitly.)
Finally, we will get an invariant by taking the coset (with respect to the ideal) represented by the state sum.
Some other similar invariants will be also outlined.

As a potential application of this invariant, we will eventually show how to get lower bounds for the {\em Matveev complexity}~\cite{Matveev:compl} of \ptwoirred\ closed 3--manifolds in terms of the invariant.
The Matveev complexity is usually difficult to compute.
Only upper bounds are easy to find (and, typically, they are very precise), while lower bounds are much more difficult to achieve.

\section{Nullhomotopic filling Dehn spheres}

Throughout this paper, all 3--manifolds are assumed to be connected.
We will mainly deal with (connected) closed 3--manifolds;
so $M$ will always denote such a closed 3--manifold.
Using the {\em Hauptvermutung}, we will freely intermingle the differentiable,
piecewise linear and topological viewpoints.

\paragraph{Dehn surfaces}
A subset $\Sigma$ of $M$ is said to be a {\em Dehn surface of $M$}~\cite{Papa}
if there exists an abstract closed surface $S$ and a transverse immersion
$f\co S\to M$ such that $\Sigma = f(S)$.

Let us fix for a while $f\co S\to M$ a transverse immersion (hence,
$\Sigma = f(S)$ is a Dehn surface of $M$). By transversality, the number
of pre-images of a point of $\Sigma$ is 1, 2 or 3; so there are three types of
points in $\Sigma$, depending on this number; they are called
{\em simple}, {\em double} or {\em triple}, respectively.
Note that the definition of the type of a point does not depend on the particular
transverse immersion $f\co S\to M$ we have chosen. In fact, the
type of a point can be also defined by looking at a regular neighbourhood (in
$M$) of the point, as shown in Fig.~\ref{fig:neigh_Dehn_surf}.
The set of triple points is denoted by $T(\Sigma)$; non-simple points are called {\em singular} and their set is denoted by $S(\Sigma)$.
\begin{figure}[ht!]
  \centerline{
  \begin{tabular}{ccc}
    \begin{minipage}[c]{3.5cm}{\small{\begin{center}
        \includegraphics{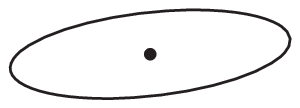}
      \end{center}}}\end{minipage} &
    \begin{minipage}[c]{3.5cm}{\small{\begin{center}
        \includegraphics{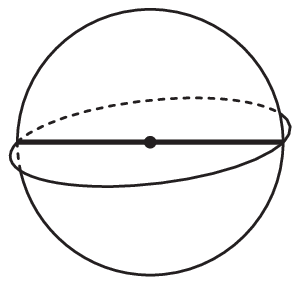}
      \end{center}}}\end{minipage} &
    \begin{minipage}[c]{3.5cm}{\small{\begin{center}
        \includegraphics{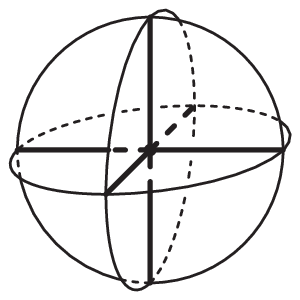}
      \end{center}}}\end{minipage} \\
    \begin{minipage}[t]{3.5cm}{\small{\begin{center}
        Simple\\point
      \end{center}}}\end{minipage} &
    \begin{minipage}[t]{3.5cm}{\small{\begin{center}
        Double\\point
      \end{center}}}\end{minipage} &
    \begin{minipage}[t]{3.5cm}{\small{\begin{center}
        Triple\\point
      \end{center}}}\end{minipage}
  \end{tabular}}
  \caption{Neighbourhoods of points (marked by thick dots) of a Dehn surface.}
  \label{fig:neigh_Dehn_surf}
\end{figure}
From now on, in all figures, triple points are always marked by thick dots and the singular set is also drawn thick.

\begin{rem}
\label{rem:euler_char}
  The topological type of the abstract surface $S$ is determined unambiguously by
  $\Sigma$.
\end{rem}

\paragraph{Filling Dehn surfaces and cubulations}
A Dehn surface $\Sigma$ of $M$ is called {\em filling}~\cite{Montesinos} if
its singularities induce a cell-decomposition of $M$; more precisely,
\begin{itemize}
\item $T(\Sigma) \neq \emptyset$,
\item $S(\Sigma) \setminus T(\Sigma)$ is made up of intervals (called {\em
    edges}),
\item $\Sigma \setminus S(\Sigma)$ is made up of discs (called {\em regions}),
\item $M \setminus \Sigma$ is made up of balls.
\end{itemize}
A {\em cubulation} of $M$ is a cell-decomposition of $M$ such that
\begin{itemize}
\item each 2--cell (called {\em face}) is glued along 4 edges,
\item each 3--cell (called {\em cube}) is glued along 6 faces arranged like the
  boundary of a cube.
\end{itemize}
Note that self-adjacencies and multiple adjacencies are allowed.
In Fig.~\ref{fig:cubul_example} we have shown a cubulation of the 3--dimensional torus
$S^1\times S^1\times S^1$ with two cubes (the identification of each pair of
faces is the obvious one, {\em i.e.}~the one without twists).
\begin{figure}[ht!]
  \centerline{\includegraphics{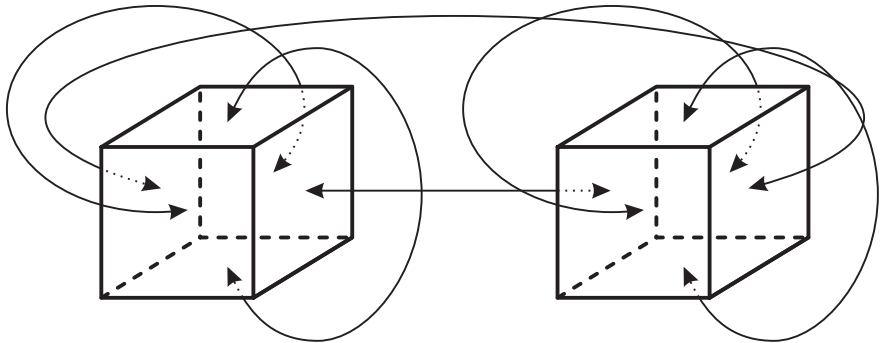}}
  \caption{A cubulation of the 3--dimensional torus $S^1\times S^1\times S^1$
    with two cubes (the identification of each pair of faces is the obvious
    one, {\em i.e.}~the one without twists).}
  \label{fig:cubul_example}
\end{figure}

The following construction is well-known (see~\cite{Aitchison-Matsumotoi-Rubinstein, Funar, Babson-Chan}, for instance).
Let $\calC$ be a cubulation of a closed 3--manifold; consider, for each cube of
$\calC$, the three squares shown in Fig.~\ref{fig:cube_to_surf};
it is quite easy to prove that the subset of $M$ obtained by gluing together
all these squares is a filling Dehn surface $\Sigma$ of $M$.
\begin{figure}[ht!]
  \centerline{\includegraphics{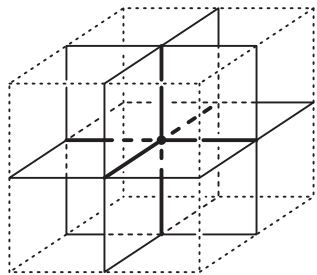}}
  \caption{Local behaviour of duality.}
  \label{fig:cube_to_surf}
\end{figure}
Conversely, a cell-decomposition $\calC$ can be constructed from a filling
Dehn surface $\Sigma$ of $M$ by considering an abstract cube for each triple
point of $\Sigma$ and by gluing the cubes together along the faces (the identification of each pair of faces is chosen by following the four germs of regions adjacent
to the respective edge of $\Sigma$); it is quite easy to prove that the cell-decomposition $\calC$ just constructed is indeed a cubulation of $M$.
The cubulation and the filling Dehn surface constructed in such a way are said
to be {\em dual} to each other.

\paragraph{Nullhomotopic filling Dehn spheres}
Let $\Sigma$ be a Dehn surface such that $\Sigma = f(S)$, where $f\co S\to M$ is a transverse immersion. If $S$ is a sphere, we will call $\Sigma$ {\em
  Dehn sphere} (this definition makes sense by Remark~\ref{rem:euler_char}).
A Dehn sphere $\Sigma$ is said to be {\em nullhomotopic} if $f$ is homotopic to a
constant map (also this definition makes sense, because it does not depend on the particular $f$ chosen).

In what follows, we will only deal with nullhomotopic filling Dehn spheres.
They are enough to study closed 3--manifolds, since they present closed 3--manifolds.
In order to prove this fact, we will use the following result, obtaining a nullhomotopic filling Dehn sphere of $M$ from any of its one-vertex triangulations.
(For an introduction to one-vertex triangulations of closed 3--manifolds, we refer the reader to~\cite{Matveev:book}.)
We point out that there are explicit constructions of nullhomotopic filling Dehn spheres~\cite{Montesinos, Vigara:present}; however, this construction is very efficient and applies to all closed $3$--manifolds.

\begin{prop}
\label{prop:tria_to_surf}
  Suppose that a closed $3$--manifold $M$ has a one-vertex triangulation $\calT$ with $c$ tetrahedra.
  Then $M$ has a nullhomotopic filling Dehn sphere with $4c$ triple points.
\end{prop}

\begin{proof}
  Consider, for each tetrahedron of $\calT$, the four triangles shown in
  Fig.~\ref{fig:tria_to_surf}.
  \begin{figure}[ht!]
    \centerline{\includegraphics{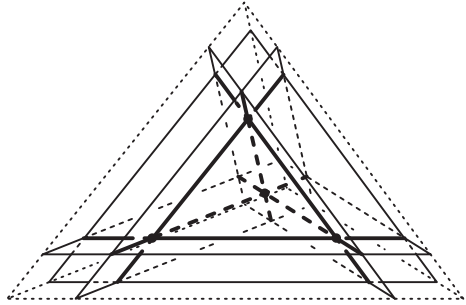}}
    \caption{Construction of a nullhomotopic filling Dehn sphere from a
      one-vertex triangulation of a closed 3--manifold.}
    \label{fig:tria_to_surf}
  \end{figure}
  The subset of $M$ obtained by gluing together all these triangles is a Dehn
  surface $\Sigma$ of $M$ with $4c$ triple points.
  The Dehn surface $\Sigma$ can be also constructed by starting with a small sphere
  whose centre is the only vertex of $\calT$ and then by inflating it.
  In order to draw a clear picture of what is happening, we have shown this construction
  in Fig.~\ref{fig:inflating} for the 2--dimensional torus; nevertheless, a genuine 3--dimensional
  picture could also be depicted for each tetrahedron of $M$, and we invite the reader to
  figure out it.
  \begin{figure}[ht!]
    \centerline{\includegraphics{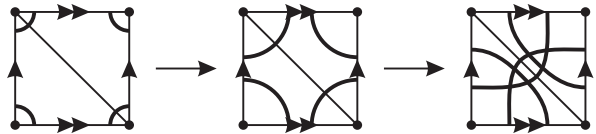}}
    \caption{Inflating of the small sphere, whose centre is the only vertex of
      $\calT$, to get $\Sigma$ (2--dimensional torus case).}
    \label{fig:inflating}
  \end{figure}
  (For the sake of completeness, we note that only closed surfaces with non-positive Euler
  characteristic have one-vertex triangulations; while all closed
  3--manifolds do.)
  Hence, we have proved that $\Sigma$ is a nullhomotopic Dehn sphere.
  Finally, it is very easy to prove that $\Sigma$ is filling, so we leave it to
  the reader.
\end{proof}

The construction described in the proof above is the dual counterpart of the well-known construction consisting in dividing a tetrahedron into four cubes~\cite{Shtanko-Shtogrin, Dolbilin-Shtanko-Shtogrin, Funar}.
However, we have not found in literature the proof that the result of the construction is the cubulation dual to a nullhomotopic filling Dehn sphere of the manifold.

We are now able to prove that nullhomotopic filling Dehn spheres present closed 3--manifolds.
\begin{teo}
\label{teo:spheres_present}
  \begin{itemize}
  \item Each closed $3$--manifold has a nullhomotopic filling Dehn sphere.
  \item If $\Sigma_1$ and $\Sigma_2$ are homeomorphic nullhomotopic filling
    Dehn spheres of closed $3$--manifolds $M_1$ and $M_2$ respectively, then $M_1$ and
    $M_2$ are also homeomorphic.
  \end{itemize}
\end{teo}

\begin{proof}
  In order to prove the first point, it is enough to start from a one-vertex
  triangulation of a closed 3--manifold (all closed 3--manifolds have
  one-vertex triangulations, as shown in~\cite{Matveev:book}, for instance)
  and then to apply Proposition~\ref{prop:tria_to_surf}.

  We are left to prove the second point.
  Let $\calC_i$ be the cubulation of $M_i$ dual to $\Sigma_i$, for $i=1,2$.
  Since $\Sigma_1$ and $\Sigma_2$ are homeomorphic, the cubulations $\calC_1$ and
  $\calC_2$ are isomorphic and hence $M_1$ and $M_2$ are homeomorphic.
\end{proof}

\section{The calculus}

In order to define an invariant using nullhomotopic filling Dehn spheres and the state-sum machinery, it is
necessary a finite calculus ({\em i.e.}~a finite set of moves, such that two
nullhomotopic filling Dehn spheres present the same closed 3--manifold if and only if
they are related to each other by a finite sequence of moves in this set).

Throughout the paper we will draw pictures to describe various modifications
of Dehn spheres; in such pictures only the portions of the Dehn spheres
involved in the modifications are drawn, while the remaining portions of the Dehn
spheres are supposed fixed.

\subsection{Vigara's calculus}

In~\cite{Vigara:calculus} Vigara has described a finite calculus with three moves (the complete proof being in~\cite{Vigara:tesi}).
Let us describe Vigara's moves in detail.
In this section $\Sigma$ will always denote a nullhomotopic filling Dehn sphere of a
closed 3--manifold $M$.

\paragraph{$\mathbf{Vig_1}$--move}
The first move is shown in Fig.~\ref{fig:vig_1} and is called {\em
  \vigone--move} (in~\cite{Vigara:calculus} it is called {\em finger move~2}).
\begin{figure}[ht!]
  \centerline{\includegraphics{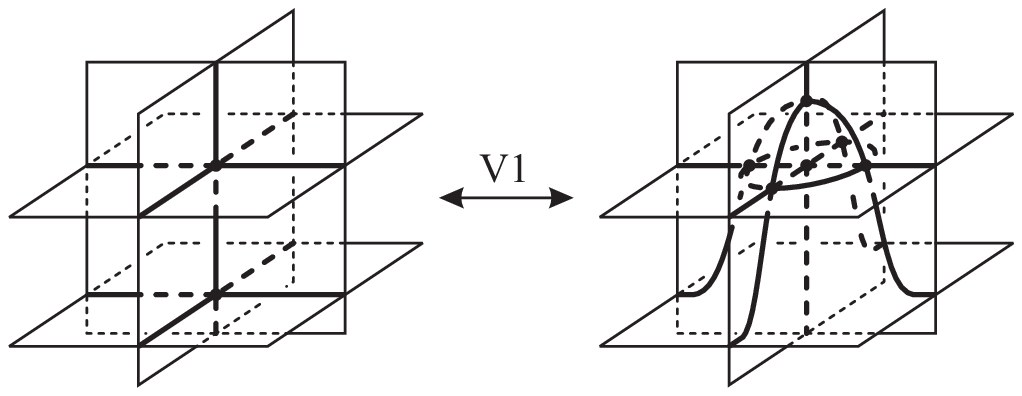}}
  \caption{\vigone--move.}
  \label{fig:vig_1}
\end{figure}
It will be called {\em positive} if it increases (by four) the number of
triple points of $\Sigma$, and {\em negative} otherwise.
Note that, if we apply a \vigone--move to $\Sigma$, the result will be another
nullhomotopic filling Dehn sphere of $M$.

\paragraph{$\mathbf{Vig_2}$--move}
The second move is shown in Fig.~\ref{fig:vig_2} and is called {\em
  \vigtwo--move} (in~\cite{Vigara:calculus} it is called {\em finger move~1}).
\begin{figure}[ht!]
  \psfrag{R1}{\small $R_1$}
  \psfrag{R2}{\small $R_2$}
  \psfrag{Ra1}{\small $R^1$}
  \psfrag{Ra2}{\small $R^2$}
  \psfrag{Ra3}{\small $R^3$}
  \psfrag{e1}{\small $e_1$}
  \psfrag{e2}{\small $e_2$}
  \centerline{\includegraphics{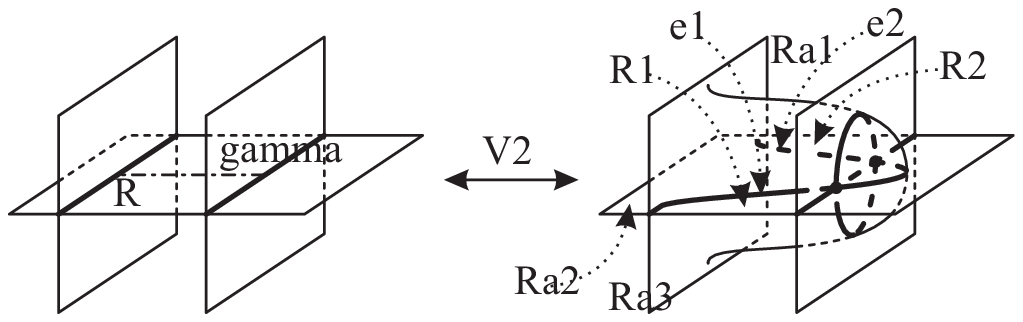}}
  \caption{\vigtwo--move.}
  \label{fig:vig_2}
\end{figure}
As above, we have {\em positive} and {\em negative} \vigtwo--moves, depending on
whether they increase or decrease (by two) the number of triple points of
$\Sigma$.
In contrast to the \vigone--move, this move is non-local, so it must be
described with some care.
A positive \vigtwo--move is determined by an arc $\gamma$ properly embedded in
a region $R$ of $\Sigma$.
The move acts on $\Sigma$ as in Fig.~\ref{fig:vig_2}, but, to define
its effect unambiguously, we must specify which pairs of regions
(out of the four ``vertical'' ones incident to $R$ at the endpoints of
$\gamma$) will become adjacent to each other after the move.
This is achieved by noting that $R$ is a disc, so its regular
neighbourhood in $M$ is a product and hence we can choose for $R$ a
transverse orientation.
Using it, at each endpoint of $\gamma$ we can tell from each other the
two ``vertical'' regions incident to $R$ as being an upper and a lower one,
and we can stipulate that the two upper regions will become incident after
the move (and similarly for the lower ones).

Obviously, a positive \vigtwo--move leads to a nullhomotopic filling Dehn sphere
of $M$.
For the negative case the situation is more complicated.
A negative \vigtwo--move may lead to a nullhomotopic Dehn sphere that is not
filling.
For instance, if $R_1$ and $R_2$ are contained in the same region, then after the negative
\vigtwo--move the ``region'' $R$ would not be a disc.
In order to avoid this loss of fillingness, we will call negative
\vigtwo--moves only those preserving fillingness.
So a negative \vigtwo--move is the inverse of a positive \vigtwo--move.
With this convention, if we apply a negative \vigtwo--move to $\Sigma$, the result will
be another nullhomotopic filling Dehn sphere of $M$.

\paragraph{$\mathbf{Vig_3}$--move}
The third move is shown in Fig.~\ref{fig:vig_3} and is called {\em
  \vigthree--move} (in~\cite{Vigara:calculus} it is called {\em saddle move}).
\begin{figure}[ht!]
  \psfrag{D}{\small $D$}
  \centerline{\includegraphics{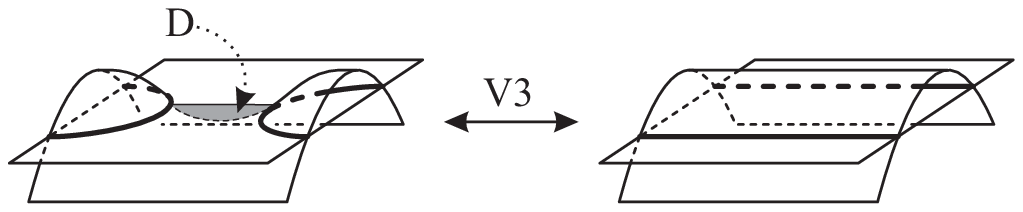}}
  \caption{\vigthree--move.}
  \label{fig:vig_3}
\end{figure}
In contrast to the other two moves, here we cannot distinguish between {\em
  positive} and {\em negative} \vigthree--moves. In fact this move is
symmetric.
As the \vigtwo--move, this move is non-local, so it must be described with some
care.
A \vigthree--move is determined by a disc (say $\Delta$) properly embedded in
a component (a ball) of $M\setminus\Sigma$, as shown in Fig.~\ref{fig:vig_3}.
The move acts on $\Sigma$ as in Fig.~\ref{fig:vig_3}, but, to define
its effect unambiguously, we must specify which pairs of regions will unite after the move.
This is achieved with the same technique as above, after noting that every
region of $\Sigma$ is a disc.

Now, we cannot conclude as done above, because a \vigthree--move defined in
such a way, when applied to $\Sigma$, leads to a nullhomotopic Dehn sphere,
which may not be filling.
In order to avoid this loss of fillingness, we will call \vigthree--moves only
those preserving fillingness.
With this convention, if we apply a \vigthree--move to $\Sigma$, the result
will be another nullhomotopic filling Dehn sphere of $M$.

\paragraph{Vigara's calculus}
We are now able to state the calculus, whose proof is outlined in~\cite{Vigara:calculus} and fully provided in~\cite{Vigara:tesi}.
\begin{teo}
[Vigara]
\label{teo:vigara}
  Let $\Sigma_1$ and $\Sigma_2$ be nullhomotopic filling Dehn spheres of closed
  $3$--manifolds $M_1$ and $M_2$, respectively.
  Then, $M_1$ and $M_2$ are homeomorphic if and only if $\Sigma_1$ and $\Sigma_2$ can be
  obtained from each other via a sequence of \vigone--, \vigtwo-- and \vigthree--moves.
\end{teo}

\subsection{The local calculus}

The \vigtwo--~and \vigthree--moves of the calculus described above are not useful in
order to define the invariant, because these moves are non-local.
More precisely, if we want to apply one of these moves, we should look not
only at the portion of the Dehn sphere involved in the move but also at the
whole Dehn sphere, to check that fillingness is preserved.
To avoid this problem, we will provide another calculus in which all
moves are local, {\em i.e.}~such that the portion of a nullhomotopic
filling Dehn sphere involved in the move tells whether the move can be applied
or not.
Let us start with the description of the moves.
As above, also in this section $\Sigma$ will denote a nullhomotopic filling Dehn
sphere of a closed 3--manifold $M$.

\paragraph{$\mathbf{T_1}$--move}
The first move is the \vigone--move described above.
In order to uniform the notation, we call it {\em
  \tone--move} and we draw it in Fig.~\ref{fig:t_1}.
\begin{figure}[ht!]
  \centerline{\includegraphics{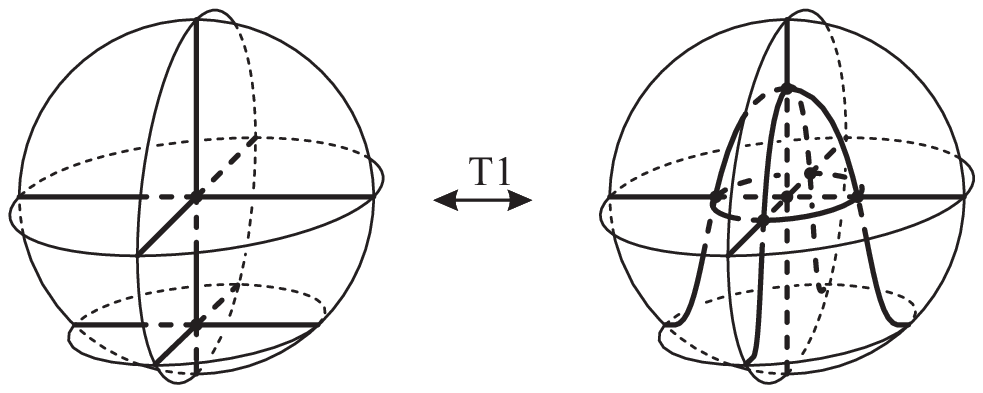}}
  \caption{\tone--move.}
  \label{fig:t_1}
\end{figure}
It will be called {\em positive} if it increases (by four) the number of
triple points of $\Sigma$, and {\em negative} otherwise.

\paragraph{$\mathbf{T_2}$--move}
The second move is shown in Fig.~\ref{fig:t_2} and is called {\em
  \ttwo--move}.
\begin{figure}[ht!]
  \centerline{\includegraphics{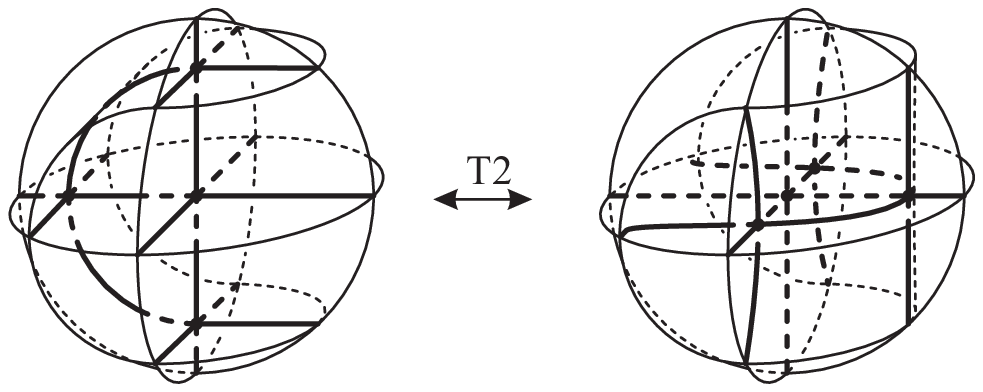}}
  \caption{\ttwo--move.}
  \label{fig:t_2}
\end{figure}
In contrast to the \tone--move, here we cannot distinguish between positive and
negative \ttwo--moves; in fact this move is symmetric.

\paragraph{$\mathbf{T_3}$--move}
The third move is shown in Fig.~\ref{fig:t_3} and is called {\em
  \tthree--move}.
\begin{figure}[ht!]
  \centerline{\includegraphics{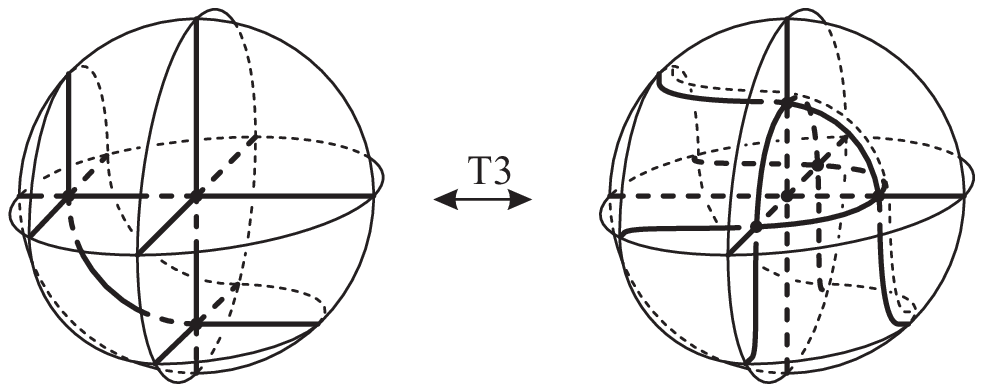}}
  \caption{\tthree--move.}
  \label{fig:t_3}
\end{figure}
It will be called {\em positive} if it increases (by two) the number of
triple points of $\Sigma$, and {\em negative} otherwise.

\paragraph{$\mathbf{T_4}$--move}
The fourth move is shown in Fig.~\ref{fig:t_4} and is called {\em
  \tfour--move}.
\begin{figure}[ht!]
  \centerline{\includegraphics{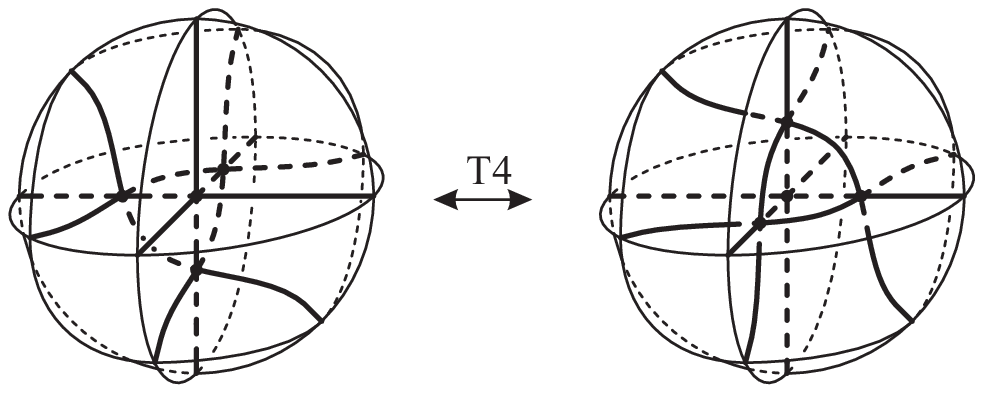}}
  \caption{\tfour--move.}
  \label{fig:t_4}
\end{figure}
As it occurs for the \ttwo--move, here we cannot distinguish between positive and
negative \tfour--moves; in fact also this move is symmetric.

\paragraph{$\mathbf{B}$--move}
The next move is slightly unnatural.
It is shown in Fig.~\ref{fig:b} and is called {\em \bubble--move}.
\begin{figure}[ht!]
  \psfrag{fD1}{\small $f(D_1)$}
  \psfrag{fD2}{\small $f(D_2)$}
  \centerline{\includegraphics{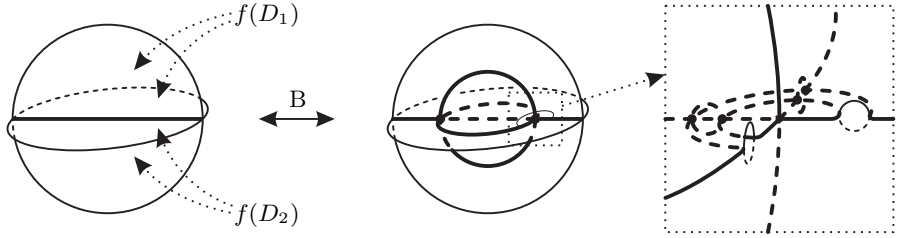}}
  \caption{\bubble--move and spiral piping.}
  \label{fig:b}
\end{figure}
It will be called {\em positive} if it increases (by six) the number of
triple points of $\Sigma$, and {\em negative} otherwise.
The configuration shown in Fig.~\ref{fig:b}-right will be called {\em spiral
  piping} and will be denoted in the figures as shown in Fig.~\ref{fig:b}-centre.

\paragraph{$\mathbf{S}$--move}
The last move is also slightly unnatural.
It is shown in Fig.~\ref{fig:s} and is called {\em \saddle--move}.
\begin{figure}[ht!]
  \centerline{\includegraphics{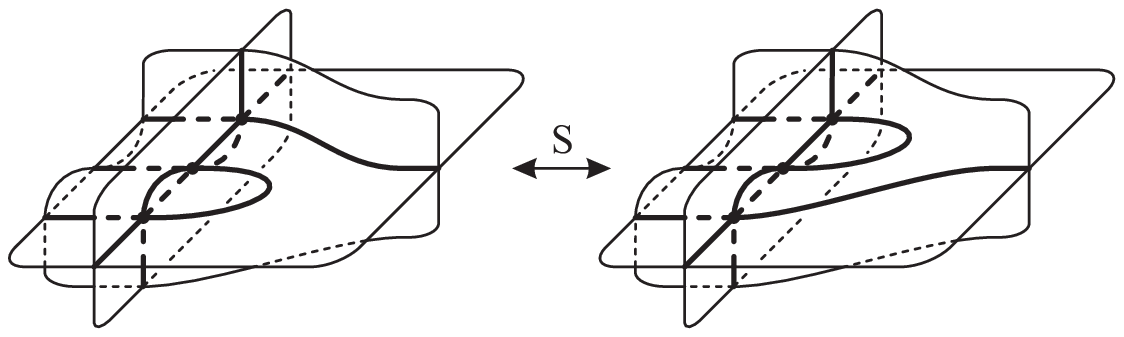}}
  \caption{\saddle--move.}
  \label{fig:s}
\end{figure}
It is a specialization of the \vigthree--move; in fact, it is a \vigthree--move
applied to a nullhomotopic filling Dehn sphere having a particular shape near
the portion involved in the \vigthree--move.
The \saddle--move is symmetric, so we cannot distinguish between positive and
negative \saddle--moves.

\paragraph{Properties of the moves}
The first property to point out is the following one.
\begin{prop}\label{prop:preserving_moves}
  If we apply a \ti--, a \bubble-- or an \saddle--move to a nullhomotopic filling
  Dehn sphere of a closed $3$--manifold $M$, the result
  will be another nullhomotopic filling Dehn sphere of $M$.
\end{prop}

\begin{proof}
  The proof consists of a straightforward case-by-case check of the
  properties in the definition of nullhomotopic filling Dehn spheres.
  Only the case of negative \bubble--moves is slightly more difficult; therefore, we
  analyse it and we invite the reader to analyse the other cases.
  
  The result of the application of a negative \bubble--move to a nullhomotopic filling
  Dehn sphere of a closed $3$--manifold $M$ is a nullhomotopic Dehn sphere $\Sigma'$ such that $S(\Sigma')\neq\emptyset$, such that
  $\Sigma'\setminus S(\Sigma')$ is made up of discs and such that $M\setminus\Sigma'$ is made up of balls.
  If we have also $T(\Sigma')\neq\emptyset$, it is very easy to prove that $\Sigma'$ is filling (we leave this proof to the reader).
  
  Hence, in order to conclude, it is enough to prove that there are triple points.
  Suppose by contradiction that $T(\Sigma')=\emptyset$.
  Let $f\co S^2\to M$ be a transverse immersion such that $\Sigma' = f(S^2)$.
  Then, the pre-image of $S(\Sigma')$ is a closed curve $\gamma$ dividing $S^2$ into two discs (recall that $\Sigma'\setminus S(\Sigma')$ is made up of discs), say $D_1$ and $D_2$.
  They lie on opposite sides of $\gamma$, so (up to symmetry) their image near a double point of $\Sigma'$ appears as in Fig.~\ref{fig:b}-left.
  An easy Euler characteristic argument proves that $M\setminus\Sigma'$ is made up of two balls.
  Now, the holonomy of the curve of double points $f(\gamma)$ must interchange the two germs of discs contained in $f(D_1)$ and must interchange the two germs of discs contained in $f(D_2)$.
  This easily implies that $M\setminus\Sigma'$ is made up of three components, which is a contradiction.
\end{proof}

It is worth noting that all these moves are local; in fact, they can be applied each time a portion of $\Sigma$ appears like one of the sides of the figures representing the moves (we should not look at the whole of $\Sigma$).

We are now able to state the calculus.

\begin{teo}
\label{teo:calculus}
  Let $\Sigma_1$ and $\Sigma_2$ be nullhomotopic filling Dehn spheres of closed
  $3$--manifolds $M_1$ and $M_2$, respectively.
  Then, $M_1$ and $M_2$ are homeomorphic if and only if $\Sigma_1$ and $\Sigma_2$ can be
  obtained from each other via a sequence of \ti--, \bubble-- and \saddle--moves.
\end{teo}

\subsection{Proof of the calculus}

This section is devoted to the proof of Theorem~\ref{teo:calculus}.

First of all, we note, by virtue of Proposition~\ref{prop:preserving_moves}, that the
result of applying a \ti--, a \bubble-- or an \saddle--move to a nullhomotopic
filling Dehn sphere of a closed 3--manifold is another nullhomotopic filling Dehn sphere
of the same manifold.
Hence, we are left to prove that such moves are enough to relate each pair of
nullhomotopic filling Dehn spheres of the same closed 3--manifold.
By Theorem~\ref{teo:vigara}, we have that each pair can be related by a
sequence of \vigi--moves, hence we need only to prove that each \vigi--move is
a composition of \ti--, \bubble-- and \saddle--moves.
For the \vigone--move there is nothing to prove, because each \vigone--move is
already a \tone--move.
With the following lemma we analyse the \vigtwo--move.

\begin{lem}
\label{lem:vigtwo}
  Each \vigtwo--move is a composition of \ti--~and \bubble--moves.
\end{lem}

\begin{proof}
  Obviously, it is enough to prove the statement only for positive \vigtwo--moves.
  So let us consider a positive \vigtwo--move between two nullhomotopic filling Dehn
  spheres (say $\Sigma$ and $\Sigma'$); see Fig.~\ref{fig:vig_2}.
  Let us first suppose that the closure in $\Sigma'$ of one of the two regions
  $R_1$ and $R_2$ is a closed disc incident to at least three triple points; see Fig.~\ref{fig:vig_2_to_new_3dim} for
  an example.
  \begin{figure}[ht!]
    \psfrag{R1}{\small $R_1$}
    \centerline{\includegraphics{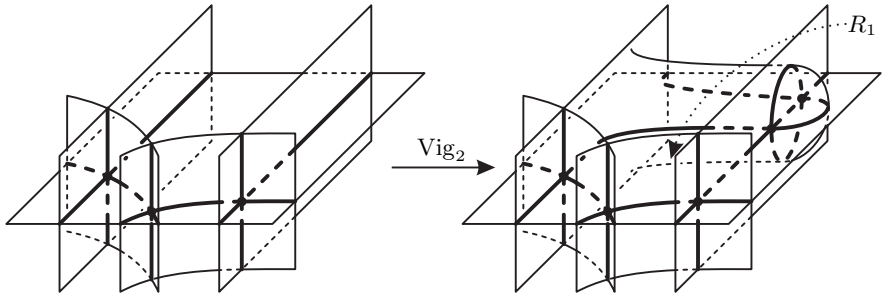}}
    \caption{A positive \vigtwo--move when the closure in $\Sigma'$ of $R_1$ is a closed
      disc incident to at least three triple points (the case of $R_2$ being symmetric).}
    \label{fig:vig_2_to_new_3dim}
  \end{figure}
  In such a case the sequence of \ti--moves shown in
  Fig.~\ref{fig:vig_2_to_new} is equivalent to the \vigtwo--move.
  For the sake of simplicity, we have shown only the singular set contained in the ``horizontal plane'' of
  Fig.~\ref{fig:vig_2_to_new_3dim}, but we invite the reader to figure out the
  3--dimensional picture.
  \begin{figure}[ht!]
    \psfrag{T1-1}{\small ${\rm T}_1^{-1}$}
    \psfrag{T3-1}{\small ${\rm T}_3^{-1}$}
    \centerline{\includegraphics{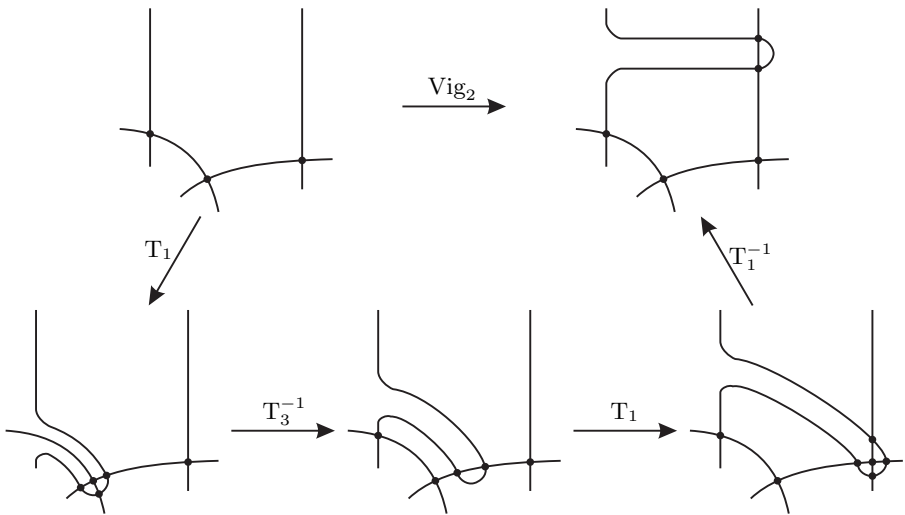}}
    \caption{The sequence of \ti--moves equivalent to the positive \vigtwo--move shown in
      Fig.~\ref{fig:vig_2_to_new_3dim}.}
    \label{fig:vig_2_to_new}
  \end{figure}
  We have shown the moves only for a particular case, but the moves are analogous in the general one; namely, they are a positive \tone--move, a negative \tthree--move, and then pairs of a positive and a negative \tone--move (depending on the number of triple points adjacent to the region).

  Suppose now that the closure of both $R_1$ and $R_2$ is not a disc, but at least one of them, say $R_1$ (the case of $R_2$ being symmetric), is incident to not less than three triple points.
  We can repeat
  the procedure above unless $R_1$ is incident to
  the edge $e_1$ more than once (see Fig.~\ref{fig:vig_2}).
  In such a case, however, we can repeat the procedure, but, when we need to pass
  along the edge $e_1$, we should add two \vigtwo--moves (a positive and a
  negative one).
  More precisely, $R_1$ may be equal to $R^i$ for $i=1,2,3$ (see Fig.~\ref{fig:vig_2}).
  We have shown in Fig.~\ref{fig:vig_2_to_new_self} the moves that are performed
  if $R_1=R^2$, the other two cases being analogous.
  \begin{figure}[ht!]
    \psfrag{Rp}{\small $R'$}
    \psfrag{T1T1}{\small $\tone + {\rm T}_1^{-1}$}
    \centerline{\includegraphics{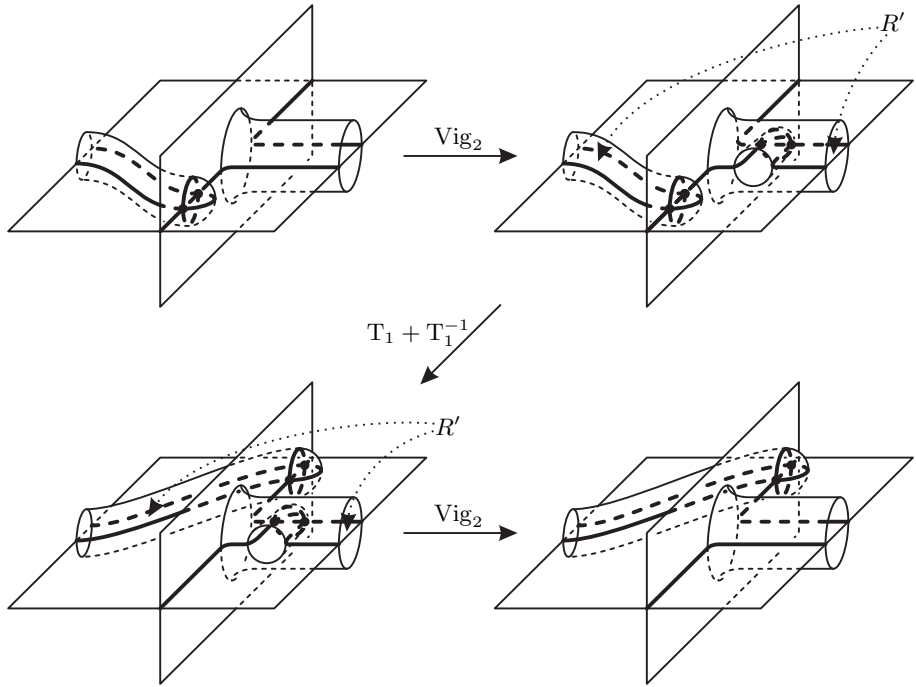}}
    \caption{If $R_1$ is incident more than once to the edge $e_1$, we must slightly modify the procedure.}
    \label{fig:vig_2_to_new_self}
  \end{figure}
  Note that the two \vigtwo--moves we have added are compositions of \ti--moves
  because the closure of the region $R'$ is a closed disc incident to three triple points (see Fig.~\ref{fig:vig_2_to_new_self}).
  
  Finally, suppose both $R_1$ and $R_2$ are incident to at most two triple points.
  In such a case, we can suitably apply a positive \bubble--move near the boundary of $R_1$ so that it comes to be incident to at least three triple points.
  Then, we can apply the \ti--moves described above and we can conclude by applying a negative \bubble--move.
  Therefore, we have proved that each \vigtwo--move is a composition of \ti--~and \bubble--moves.
\end{proof}

Before analysing the \vigthree--move, we introduce some notation and we prove
two technical lemmas.

\paragraph{Passing through spiral pipings}
The move shown in Fig.~\ref{fig:spiral_move} is called {\em spiral piping
  passing move}.
\begin{figure}[ht!]
  \centerline{\includegraphics{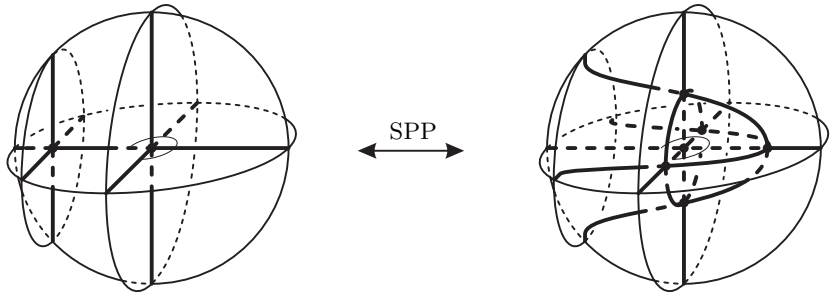}}
  \caption{Spiral piping passing move.}
  \label{fig:spiral_move}
\end{figure}
It will be called {\em positive} if it increases (by four) the number of
triple points of the Dehn sphere, and {\em negative} otherwise.
It is worth noting that this move is only a particular case of the 
{\em piping passing move} of~\cite{Vigara:calculus}.

\begin{lem}
\label{lem:spiral_move}
  Each spiral piping passing move is a composition of \ti--moves.
\end{lem}
A more general result is proved in~\cite{Vigara:calculus}; however, we prove
it here for the sake of completeness.
\begin{proof}
  Obviously, it is enough to prove the statement only for positive spiral piping
  passing moves.
  The sequence of \ti--moves shown in
  Fig.~\ref{fig:spiral_move_2dim} is equivalent to the positive spiral piping passing move.
  \begin{figure}[ht!]
    \psfrag{T3-1}{\small ${\rm T}_3^{-1}$} 
    \centerline{\includegraphics{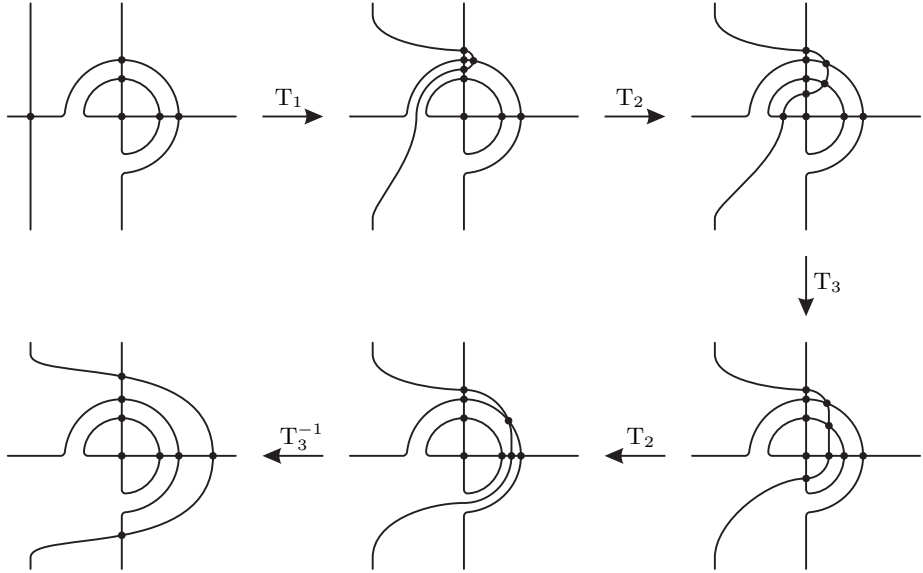}}
    \caption{The sequence of \ti--moves equivalent to the positive spiral piping passing move.}
    \label{fig:spiral_move_2dim}
  \end{figure}
  For the sake of simplicity, we have shown only the singular set contained in the ``horizontal plane'' of
  Fig.~\ref{fig:spiral_move}, but we invite the reader to figure out the
  3--dimensional picture.
\end{proof}

\paragraph{The wall}
In order to simplify the proof, we will need another move, which will turn
out to be a composition of \ti--~and \bubble--moves.

Let $\Sigma$ be a nullhomotopic filling Dehn sphere of a closed 3--manifold $M$.
Let $D$ be a closed disc embedded in $M$ such that
\begin{itemize}
\item $\partial D\subset\Sigma$,
\item $\int(D)\cap\Sigma = \emptyset$,
\item $\partial D\cap T(\Sigma) = \emptyset$,
\item $\partial D\cap S(\Sigma) \neq \emptyset$,
\item $\#(\partial D\cap S(\Sigma)) \geqslant 2$.
\end{itemize}
Let $B$ be a small regular neighbourhood of $D$ in $M$.
Obviously, $B$ is a ball whose boundary $\partial B$ is a sphere intersecting
some edges of $\Sigma$.
Let $\Sigma_{D,p}$ be the Dehn surface obtained from
$\Sigma\cup\partial B$ by replacing a small neighbourhood of a triple point $p$
(of $\Sigma\cup\partial B$) contained in $\partial B$ with a spiral piping as
shown in Fig.~\ref{fig:wall}.
\begin{figure}[ht!]
  \centerline{\includegraphics{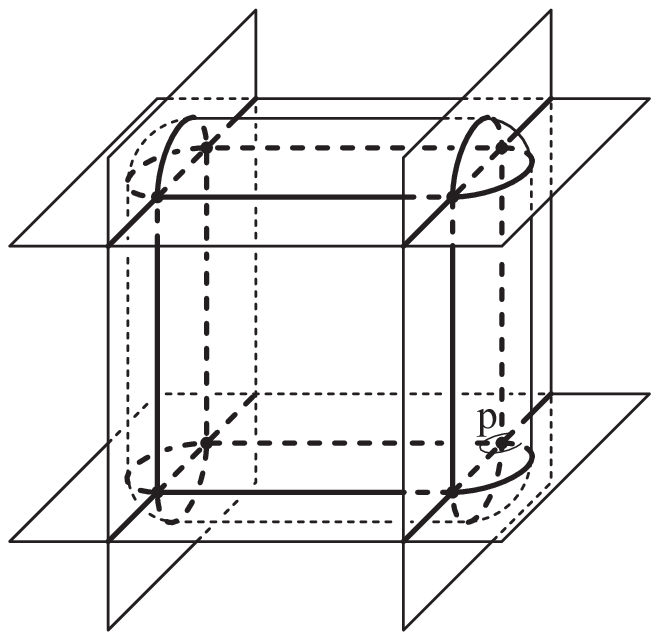}}
  \caption{A wall with $\#(\partial D\cap S(\Sigma))=4$.}
  \label{fig:wall}
\end{figure}
Note that the choice of the triple point $p$ is arbitrary, but the orientation
of the spiral piping is not.
Note also that we can think of $p$ both as a (double) point of $\Sigma$ and as a
(triple) point of $\Sigma_{D,p}$.
It is very easy to check that $\Sigma_{D,p}$ is a nullhomotopic filling Dehn
sphere of $M$.

The configuration shown in Fig.~\ref{fig:wall} is said {\em wall of
  $\Sigma_{D,p}$ with respect to the triple $(\Sigma,D,p)$}, and the move
between $\Sigma$ and $\Sigma_{D,p}$ is said {\em positive \wall--move}.
Note that this move is similar to the \bubble--move, but it is
non-local and actually there are infinitely many different \wall--moves depending on the number of singular points in $\partial D$.
A {\em negative \wall--move} is the inverse of a positive \wall--move.
Note that, in order to apply a negative \wall--move to a nullhomotopic
filling Dehn sphere of $M$, we should essentially check that the result
(which is anyway a nullhomotopic Dehn sphere of $M$) is filling.

We have the following lemma, whose (long and technical) proof will be postponed for a while.
\begin{lem}\label{lem:wall}
  Each \wall--move is a composition of \ti--~and \bubble--moves.
\end{lem}

\paragraph{Proof of the calculus}
After stating Lemma~\ref{lem:wall}, we conclude the proof of the calculus.

\begin{proof}[Proof of Theorem~\ref{teo:calculus}]
  We have already noted in the first part of this section that, by
  Proposition~\ref{prop:preserving_moves}, the result of applying a \ti--, a
  \bubble-- or an \saddle--move to a nullhomotopic filling Dehn sphere of a
  closed 3--manifold is another nullhomotopic filling Dehn sphere of the same manifold.
  Hence, we are left to prove that such moves are enough to relate each pair of
  nullhomotopic filling Dehn spheres of the same closed 3--manifold.
  By virtue of Theorem~\ref{teo:vigara}, we need only to prove that each \vigi--move is
  a composition of \ti--, \bubble-- and \saddle--moves.
  For the \vigone--move there is nothing to prove, because each \vigone--move is
  already a \tone--move.
  By using Lemma~\ref{lem:vigtwo}, we have that each \vigtwo--move is a composition of
  \ti--~and \bubble--moves.
  Hence, we are left to prove that each \vigthree--move is a composition of
  \ti--, \bubble-- and \saddle--moves.

  Let us consider a \vigthree--move between two nullhomotopic filling Dehn
  spheres (say $\Sigma$ and $\Sigma'$); see Fig.~\ref{fig:vig_3}.
  The idea is to modify the portion of $\Sigma$ involved in the move
  (shown in Fig.~\ref{fig:vig_3}-left) via some \ti--~and \wall--moves in order to apply
  an \saddle--move, and then to reconstruct the portion of $\Sigma'$ involved in
  the move (shown in Fig.~\ref{fig:vig_3}-right).
  So we start by applying two positive \wall--moves and two positive \ti--moves, as shown in
  Fig.~\ref{fig:pre_saddle}.
  \begin{figure}[ht!]
    \psfrag{T1T3}{\small $\tone + \tthree$}
    \centerline{\includegraphics[width=12cm]{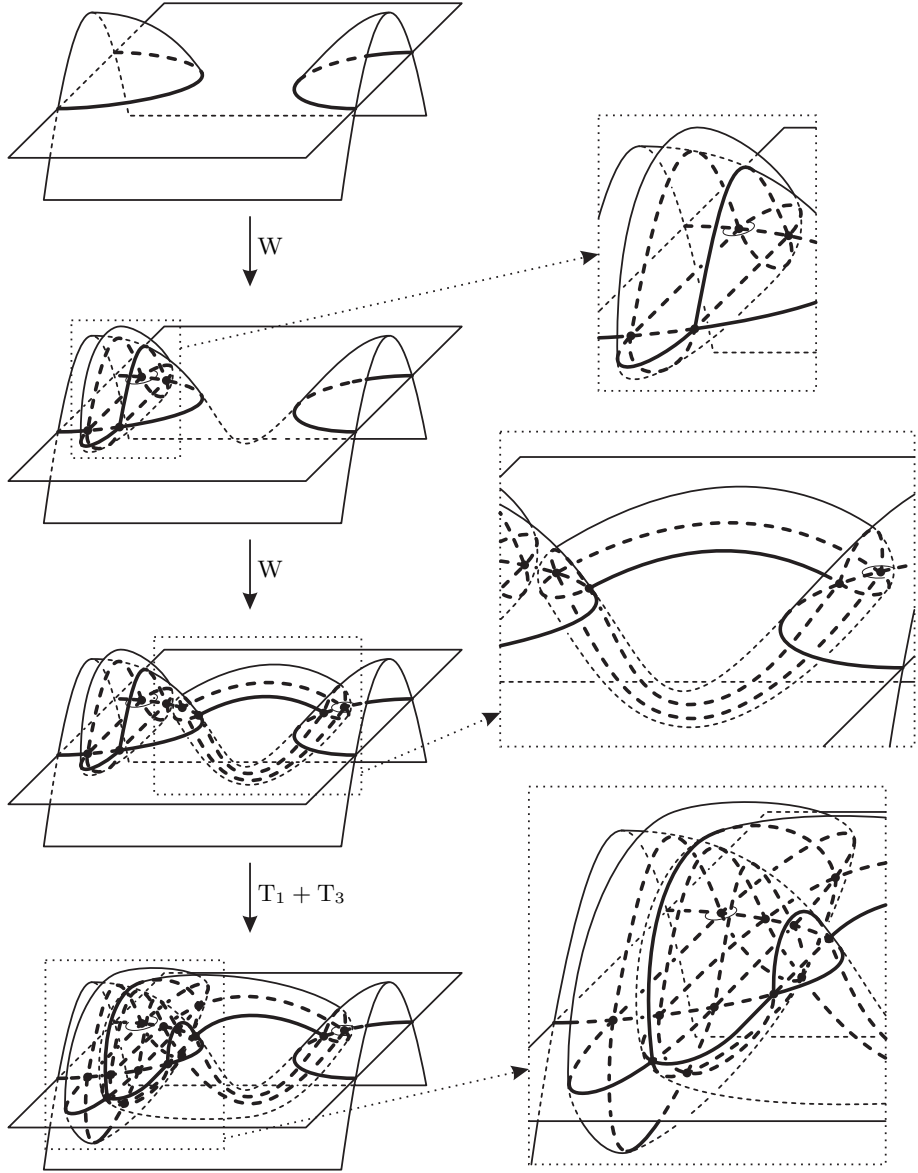}}
    \caption{\ti--, \wall-- and \saddle--moves equivalent to the \vigthree--move (first
      part).}
    \label{fig:pre_saddle}
  \end{figure}
  Then, we can apply an \saddle--move.
  Finally, we apply two negative \ti--moves and two negative \wall--moves, as shown in
  Fig.~\ref{fig:post_saddle}.
  \begin{figure}[ht!]
    \psfrag{T3T1}{\small ${\rm T}_3^{-1} + {\rm T}_1^{-1}$}
    \psfrag{invW}{\small $\wall^{-1}$}
    \centerline{\includegraphics[width=12cm]{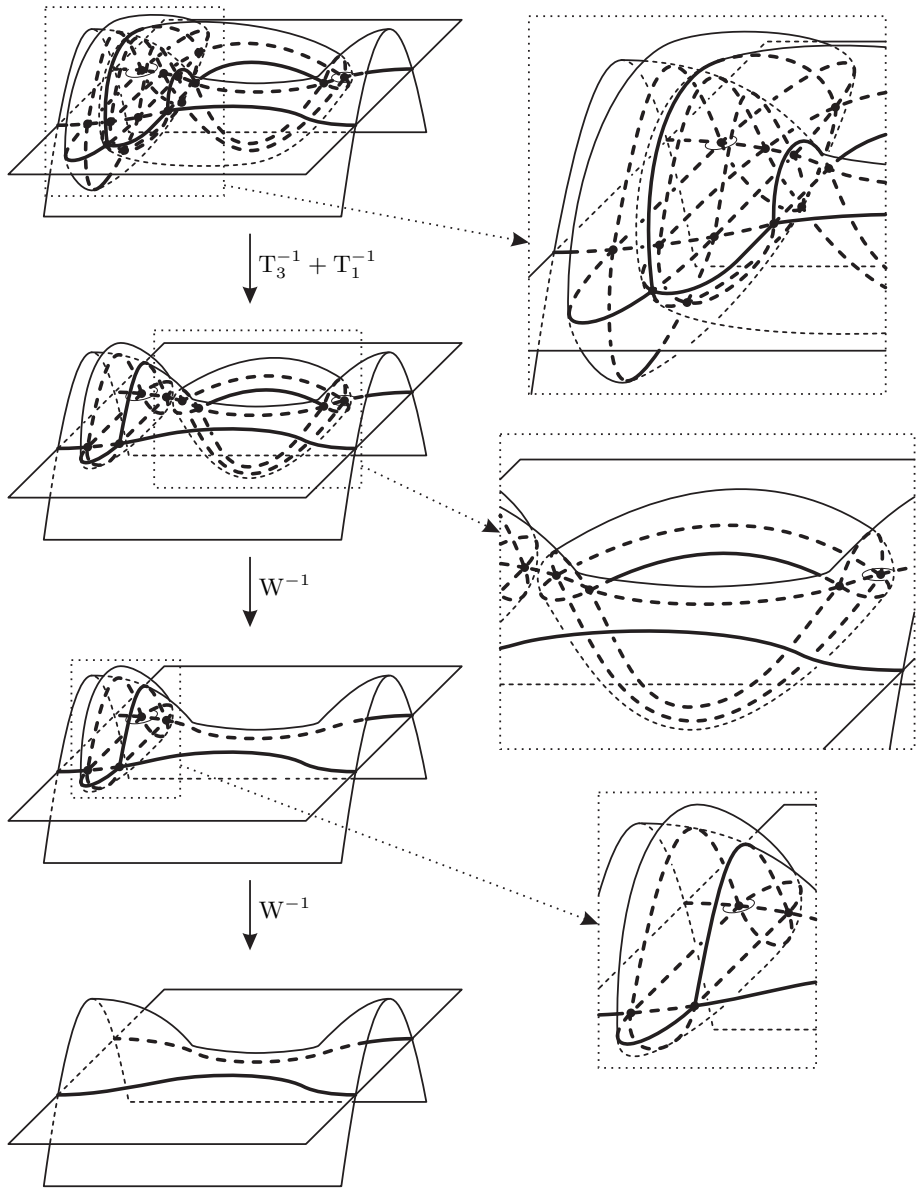}}
    \caption{\ti--, \wall-- and \saddle--moves equivalent to the \vigthree--move (second
      part).}
    \label{fig:post_saddle}
  \end{figure}
  (Note that indeed the negative \wall--moves preserve fillingness because the \vigthree--move does.)
  Now, we have proved that each \vigthree--move is a composition of \ti--, \wall-- and \saddle--moves; so, in order to conclude the proof, it is enough to note that each \wall--move is a composition of \ti--~and \bubble--moves by virtue of Lemma~\ref{lem:wall}.
\end{proof}

\paragraph{The very technical proof}
We conclude this section with the proof of the fact that each \wall--move is
a composition of \ti--~and \bubble--moves.
We warn the reader that this proof is quite long and technical, so it can be skipped at first.

\begin{proof}[Proof of Lemma~\ref{lem:wall}]
  First of all, we note that, since a negative \wall--move is the inverse of a positive
  \wall--move, it is enough to prove that each positive \wall--move is a
  composition of \ti--~and \bubble--moves.
  As a matter of fact, it is enough to prove that each positive \wall--move is a
  composition of \ti--, \bubble-- and \vigtwo--moves, because each \vigtwo--move
  is a composition of \ti-- and \bubble--moves by virtue of Lemma~\ref{lem:vigtwo}.
  Hence, let us consider a positive \wall--move between two nullhomotopic filling Dehn
  spheres $\Sigma$ and $\Sigma_{D,p}$ of a closed 3--manifold $M$, where $D$ is a closed
  disc embedded in $M$ such that
  $\partial D\subset\Sigma$,
  $\int(D)\cap\Sigma = \emptyset$,
  $\partial D\cap T(\Sigma) = \emptyset$,
  $\partial D\cap S(\Sigma) \neq \emptyset$,
  $\#(\partial D\cap S(\Sigma)) \geqslant 2$.
  In the figures below we will draw the case where $\#(\partial D\cap S(\Sigma))=4$ as in Fig.~\ref{fig:wall}, the other cases being analogous.
  Let us consider a small disc $D'$ near $p$ as shown in
  Fig.~\ref{fig:dprime}.
  \begin{figure}[ht!]
    \centerline{\includegraphics{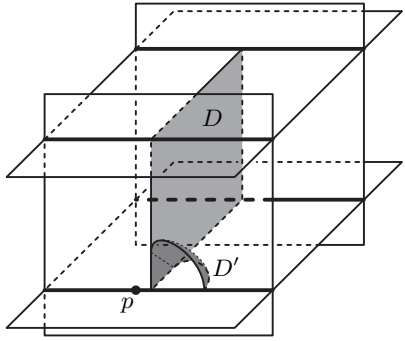}}
    \caption{A small disc $D'$ near $p$.}
    \label{fig:dprime}
  \end{figure}
  We choose $D'$ so that $D\cap D'$ is the small
  triangle shown in Fig.~\ref{fig:dprime}.
  Let $C$ be the connected component of $M\setminus\Sigma$ containing
  $\int(D)$.
  Since $\Sigma$ is filling, we have that $C$ is a ball and that it is divided by the two discs
  $D$ and $D'$ into three balls, one of which, say $C'$, is not incident to the triangle
  $D\cap D'$.
  The idea of the proof is to create a small wall with respect to the triple
  $(\Sigma,D',p)$ and then to move it through the ball
  $C'$.

  Let us start by creating the small wall.
  Let us call $e$ the edge of $\Sigma$ containing the triple point $p$.
  It is divided into two parts by the closure of the triangle $D\cap D'$; let
  us call $e'$ the one intersecting twice the closure of $D'$, and $p'$ the
  triple point at the end of $e'$; see Fig.~\ref{fig:end_point}.
  \begin{figure}[ht!]
    \psfrag{ep}{\small $e'$}
    \centerline{\includegraphics{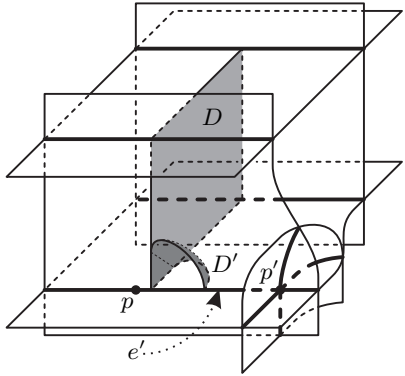}}
    \caption{The choice of the triple point $p'$.}
    \label{fig:end_point}
  \end{figure}
  (Note that $e'$ may intersect more than once the closure of $D$, but this
  does not affect the procedure).
  In order to get $\Sigma_{D',p}$ from $\Sigma$ we apply the moves shown in
  Fig.~\ref{fig:creating_wall}.
  \begin{figure}[ht!]
    \psfrag{T1-1}{\small ${\rm T}_1^{-1}$}
    \centerline{\includegraphics{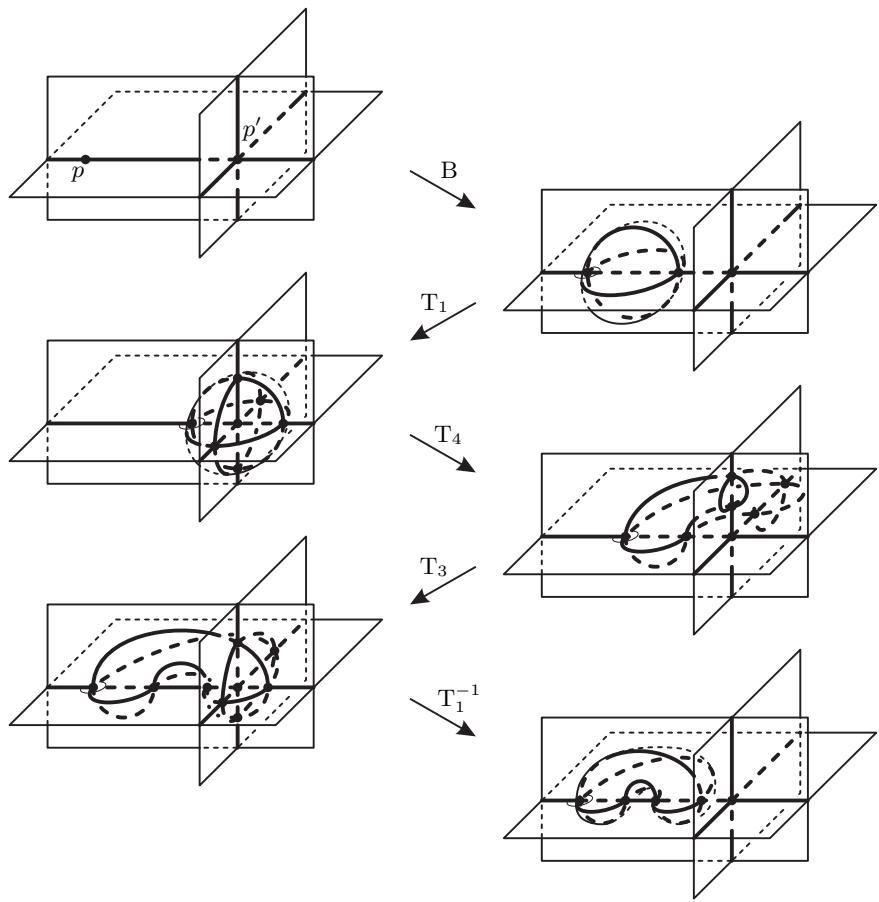}}
    \caption{The creation of the small wall.}
    \label{fig:creating_wall}
  \end{figure}

  We have created the small wall; we now need to move it through the ball
  $C'$.
  Note that, if the two Dehn spheres $\Sigma_{D,p}$ and $\Sigma_{D',p}$ are
  isotopic (this is not the case if $\#(\partial D\cap S(\Sigma))=4$, but this may occur if $\#(\partial D\cap S(\Sigma))=2$), we have done; so we suppose they are not isotopic.

  Firstly, {\em let us suppose that the closure of $C'$ is a closed ball}.
  We can move the disc $D'$ through the ball $C'$ via an isotopy
  keeping fixed the triangle $D\cap D'$ (see Fig.~\ref{fig:disc_moving} for
  an example).
  \begin{figure}[ht!]
    \centerline{\includegraphics{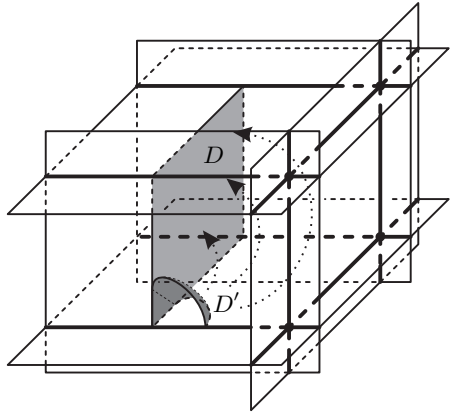}}
    \caption{Moving the disc $D'$ through $C'$.}
    \label{fig:disc_moving}
  \end{figure}
  If we consider also the trivalent graph $S(\Sigma)\cap\partial C'$, a simple
  general position argument tells us that the isotopy can be
  substituted by {\em \onetotwo--moves} and {\em \zerototwo--moves}; see
  Fig.~\ref{fig:3val_moves}.
  \begin{figure}[ht!]
    \centerline{\includegraphics{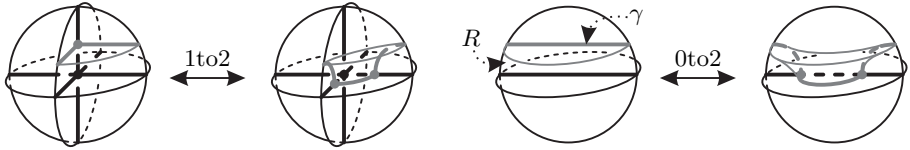}}
    \caption{\onetotwo--move (left) and \zerototwo--move (right).}
    \label{fig:3val_moves}
  \end{figure}
  (In the figures the disc we are moving is drawn in grey.)

  We now prove that each \zerototwo--move is a composition of \onetotwo--~and
  \vigtwo--moves.
  Consider a \zerototwo--move (see Fig.~\ref{fig:3val_moves}-right).
  Let $R$ be the connected component of
  $\Sigma \setminus \big(S(\Sigma)\cup\partial D'\big)$
  that is divided in two after the \zerototwo--move, and $\gamma$
  the arc of $\partial D'$ that is moved after the \zerototwo--move (see
  again 	Fig.~\ref{fig:3val_moves}-right).
  Note that the boundary of $C'$ appears near the portion of it involved in
  the move as in Fig.~\ref{fig:02to12}-left, because $R$ is a disc and the
  endpoints of $\gamma$ are double points of $\Sigma$.
  \begin{figure}[ht!]
    \psfrag{V2i}{\small ${\rm Vig}_2^{-1}$}
    \centerline{\includegraphics{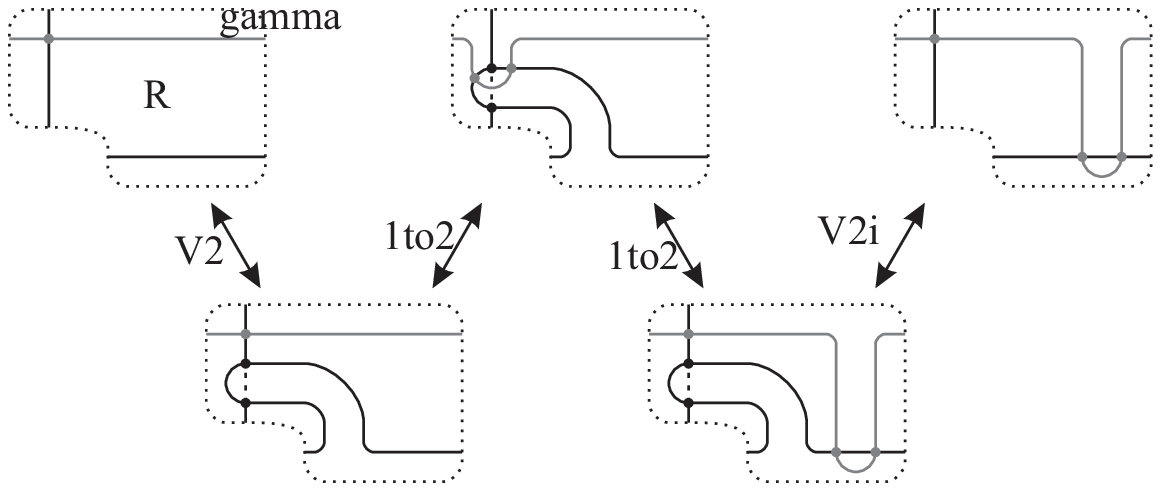}}
    \caption{Each \zerototwo--move is a composition of \onetotwo--~and
      \vigtwo--moves.}
    \label{fig:02to12}
  \end{figure}
  The \onetotwo--~and \vigtwo--moves shown in Fig.~\ref{fig:02to12} are equivalent to the
  \zerototwo--move.
  
  We have proved that the isotopy of $D'$ above can be substituted by \onetotwo--~and
  \vigtwo--moves.
  We now consider the small wall. We cannot apply a \onetotwo--move to $\Sigma_{D',p}$ at once.
  But, if we substitute each \onetotwo--move with a \tone--~and a \tthree--move as shown in
  Fig.~\ref{fig:12toti}, we get a sequence of
  \ti--~and \vigtwo--moves transforming $\Sigma_{D',p}$ into $\Sigma_{D,p}$.
  \begin{figure}[ht!]
    \centerline{\includegraphics{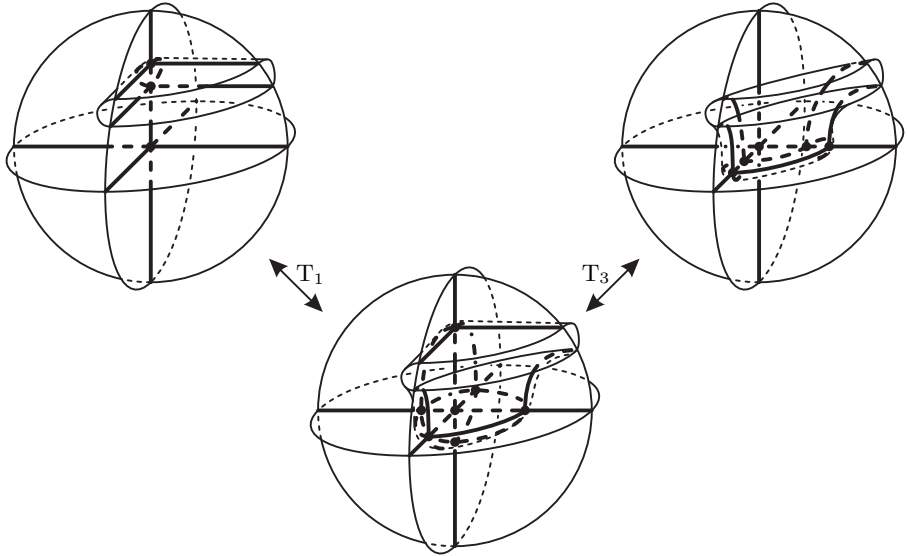}}
    \caption{Each \onetotwo--move is substituted with two \ti--moves.}
    \label{fig:12toti}
  \end{figure}
  (Note that indeed the negative \vigtwo--move preserves fillingness.)
  Hence, we have proved the statement if the closure of $C'$ is a closed ball.

  Consider now the general case; namely, {\em we suppose no more that the closure of $C'$
  is a closed ball}.
  We need to prove that $\Sigma_{D,p}$ can be obtained from
  $\Sigma_{D',p}$ via \ti--, \bubble-- and \vigtwo--moves.
  The technique is analogous to that used in the case already analysed when
  the closure of $C'$ is a closed
  ball, but here we must prepare the Dehn sphere
  $\Sigma_{D',p}$ before moving the small wall.
  For the sake of simplicity, we continue calling $\Sigma_{D',p}$ all
  Dehn spheres obtained throughout the procedure.

  First of all, we apply a positive \bubble--move, a positive spiral piping passing move
  (which is a composition of \ti--moves by virtue of Lemma~\ref{lem:spiral_move}) and a
  \tfour--move; see Fig.~\ref{fig:dprime_emb}.
  \begin{figure}[ht!]
    \centerline{\includegraphics{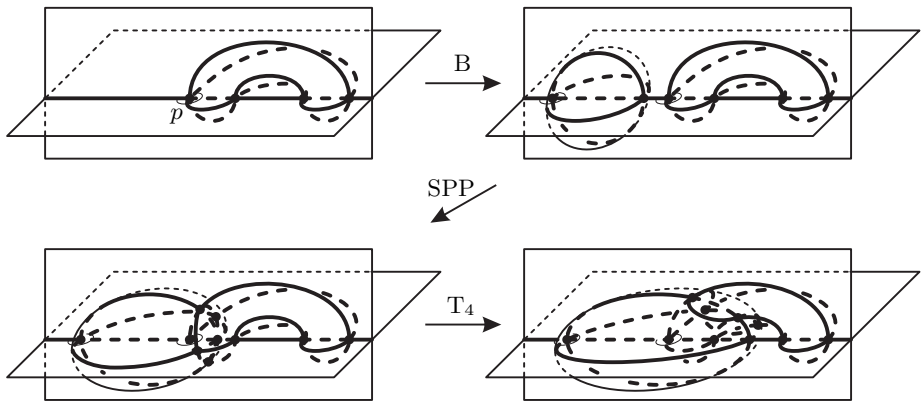}}
    \caption{Preparation before moving the small wall (first part).}
    \label{fig:dprime_emb}
  \end{figure}
  Then, we apply positive \bubble--moves and positive \vigtwo--moves near the boundary of $D$,
  as shown in Fig.~\ref{fig:d_emb}.
  \begin{figure}[ht!]
 	\psfrag{3xB}{\small $3\times\bubble$}
 	\psfrag{3xV2}{\small $3\times\vigtwo$}
    \centerline{\includegraphics{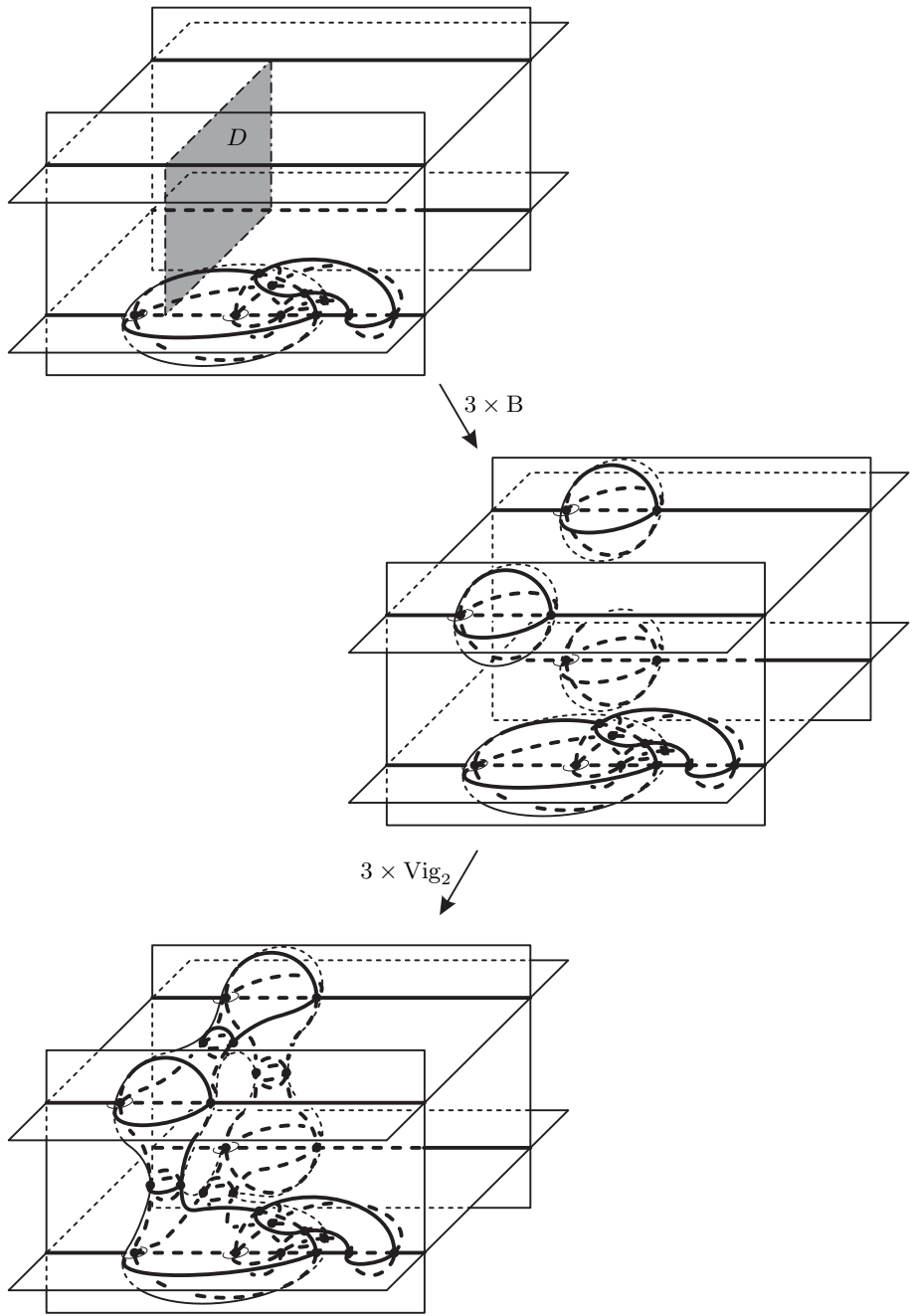}}
    \caption{Preparation before moving the small wall (second part).}
    \label{fig:d_emb}
  \end{figure}
  Afterwards, we apply two positive \tone--~and two \tfour--moves (the result is
  shown in Fig.~\ref{fig:pre_moving}).
  \begin{figure}[ht!]
    \centerline{\includegraphics[width=12cm]{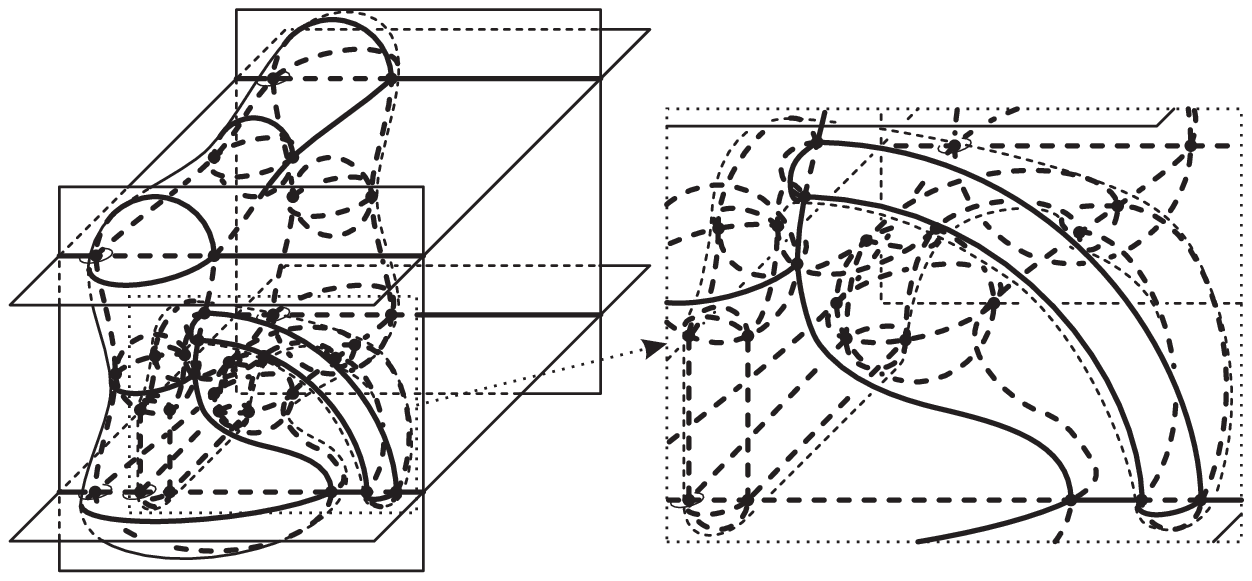}}
    \caption{Preparation before moving the small wall (third part).}
    \label{fig:pre_moving}
  \end{figure}
  Finally, we apply three pairs of a \tone--~and a \tthree--move, as we have
  done when we have replaced the \onetotwo--moves (the result is shown in
  Fig.~\ref{fig:pre_moving_no_self}).
  \begin{figure}[ht!]
    \centerline{\includegraphics{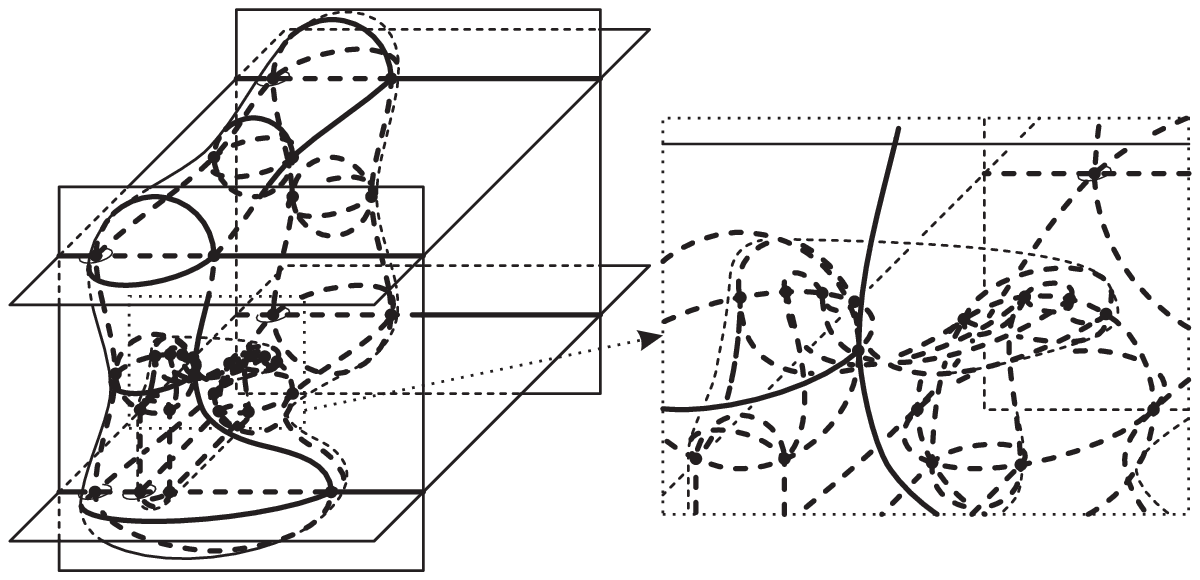}}
    \caption{Preparation before moving the small wall (fourth part).}
    \label{fig:pre_moving_no_self}
  \end{figure}
  Note that here we have used the hypothesis that $\#(\partial D\cap S(\Sigma)) \geqslant 2$.

  At this point, we look at the closure of $C'$.
  It can be thought as an abstract closed ball with some self-identifications
  on the boundary.
  In order to simplify such identifications, we apply a \bubble--move and a
  positive \tone--move for each triple point of $\Sigma_{D',p}$ where we have a
  self-adjacency of the closure of $C'$; then, we apply a positive spiral piping passing move (which is
  a composition of \ti--moves by virtue of Lemma~\ref{lem:spiral_move}) or a positive
  \tone--move for each edge of $\Sigma_{D',p}$ where we have a self-adjacency
  of the closure of $C'$; see Fig.~\ref{fig:cprime_emb}.
  \begin{figure}[ht!]
  	\psfrag{T1SPP}{\small $\tone/{\rm SPP}$}
    \centerline{\includegraphics{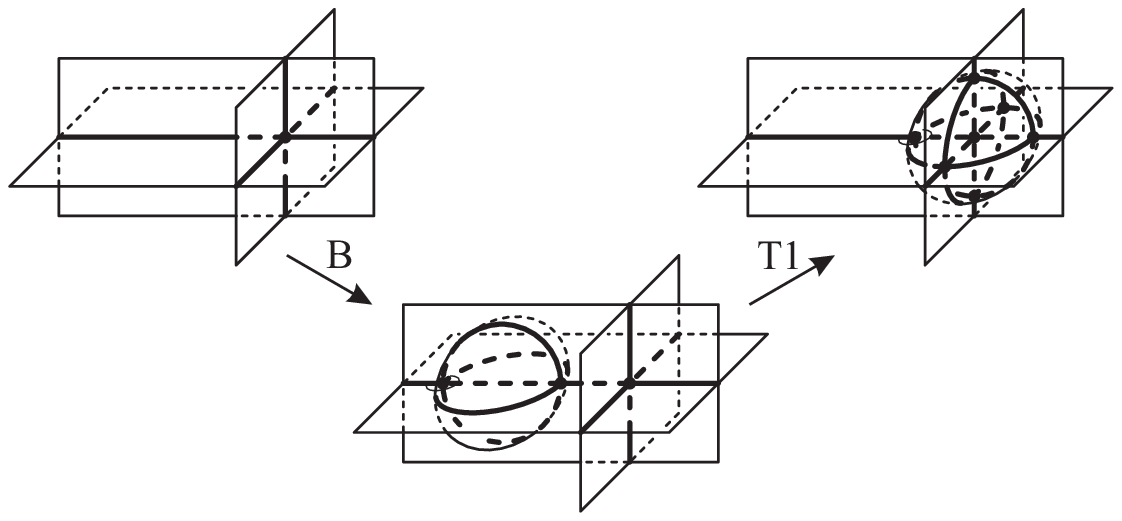}}
    \centerline{\includegraphics{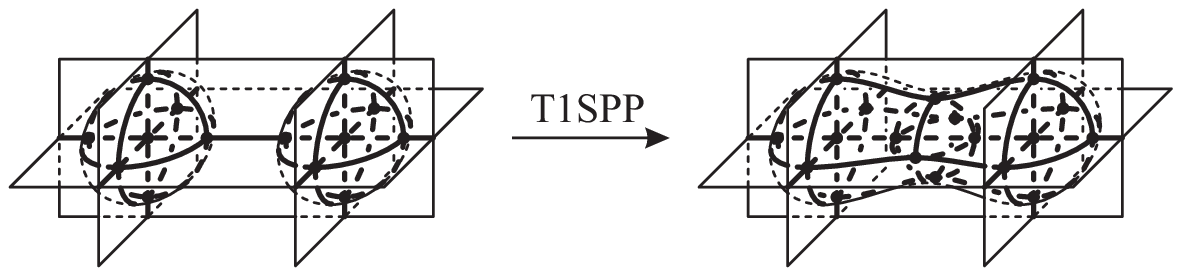}}
    \caption{Moves to simplify the self-adjacency of the closure of $C'$ near each triple point
      (above) and edge (below) of $\Sigma_{D',p}$ where we have a
      self-adjacency of the closure of $C'$.
      (For the edge case, one of the triple points of each
      sphere constructed above is replaced by a spiral piping.)}
    \label{fig:cprime_emb}
  \end{figure}
  Note that we must avoid that two spiral pipings are on the same edge, for otherwise we can apply neither the positive spiral piping passing move nor the positive \tone--move;
  this
  can be achieved by considering that, since $S(\Sigma_{D',p})$ is a hexavalent graph,
  we can choose for each triple point
  an edge adjacent to it so that each edge is chosen for at most one triple point
  (this holds for each graph containing no connected component
  that is a tree).
  Note also that the small wall is not affected by these moves.
  Now, the self-identifications are along discs, each of which is contained in a region of $\Sigma$ and can be thought as small as we
  want (with respect to $C'$).
  
  We can finally move the small wall through the ball $C'$, as we have done above
  in the case when the closure of $C'$ is a ball.
  We just must be careful because the closure of $C'$ is not a ball; however,
  the isotopy of $D'$ can be chosen so that $D'$ is always incident to one
  side of each self-adjacency disc at most.
  With such an isotopy we can repeat the procedure done above in the case
  when the closure of $C'$ is a ball.
  The result is shown in Fig.~\ref{fig:after_moving}.
  \begin{figure}[ht!]
    \centerline{\includegraphics{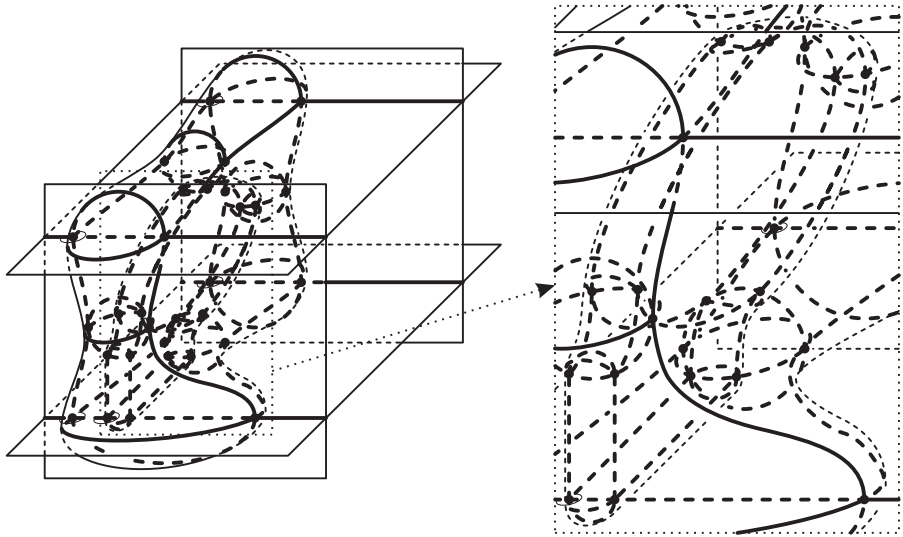}}
    \caption{Configuration after moving the small wall through $C'$.}
    \label{fig:after_moving}
  \end{figure}

  In order to conclude, we firstly apply, in reverse order, the moves done above to simplify the self-adjacency of the closure of $C'$ (see Fig.~\ref{fig:cprime_emb}).
  Afterwards, we apply some \ti--moves to put the wall in the right position; the result is shown in Fig.~\ref{fig:wall_done}.
  \begin{figure}[ht!]
	\psfrag{p1}{\small $p_1$}
	\psfrag{p2}{\small $p_2$}
	\psfrag{p3}{\small $p_3$}
	\psfrag{p4}{\small $p_4$}
	\psfrag{p5}{\small $p_5$}
	\psfrag{p6}{\small $p_6$}
	\psfrag{p7}{\small $p_7$}
    \centerline{\includegraphics{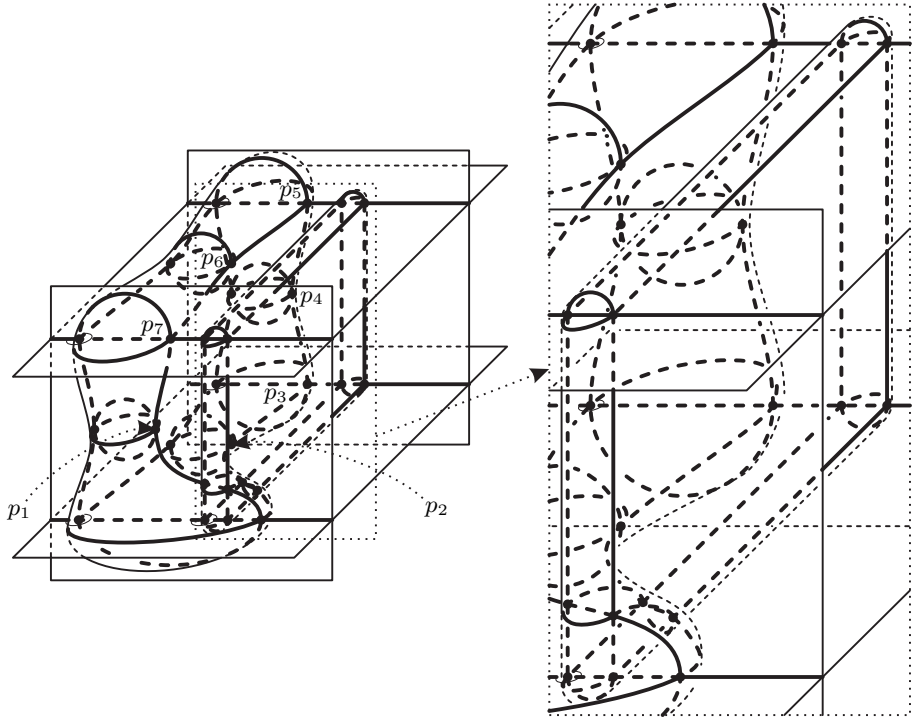}}
    \caption{The wall is in the right position.}
    \label{fig:wall_done}
  \end{figure}
  Note that there are many possibilities for accomplishing this task;
  for instance, one is to apply
  \begin{itemize}
    \item
    twice a positive and a negative \tthree--move
    (to move the wall across the triple points $p_1$ and $p_2$),
    \item
    twice a positive \tone--~and a negative \tthree--move
    (to move the wall across $p_4$ and $p_6$),
    \item
    thrice a positive \tthree--~and a negative \tone--move
    (to move the wall across $p_3$, $p_5$ and $p_7$);
  \end{itemize}
  see Fig.~\ref{fig:wall_done}.
  If we have $\#(\partial D\cap S(\Sigma)) \neq 4$, the situation is analogous; in fact, only the number of pairs of moves changes.
  Finally, we apply the moves shown in Fig.~\ref{fig:dprime_emb}
  and~\ref{fig:d_emb} in reverse order.
  The result is $\Sigma_{D,p}$, which has been obtained from $\Sigma$ via
  \ti--~and \bubble--moves.
  So the proof is complete.
\end{proof}

\section{The invariant}

After establishing Theorem~\ref{teo:calculus}, we are in a position to
define the invariant.

Let $\Sigma$ be a nullhomotopic filling Dehn sphere of a closed 3--manifold $M$.
Recall that $T(\Sigma)$ is the set of triple points of $\Sigma$ and
that $\Sigma\setminus S(\Sigma)$ ({\em i.e.}~the set of simple points)
is made up of disjoint discs.
Let us call $C(\Sigma)$ the class of these discs.

Moreover, let $\calF$ be a finite set consisting of $m>1$ elements
(called {\em colours}).
An $\calF$--{\em colouring} of $\Sigma$ is a map
$\varphi \co C(\Sigma) \to \calF$.
The set of all $\calF$--colourings of $\Sigma$ is denoted by
$\Phi_{\calF}(\Sigma)$.
If $\widetilde C\subset C(\Sigma)$, we denote by $\Phi_{\calF}(\widetilde C)$
the set of the maps $\varphi \co \widetilde C \to {\calF}$.
Note that $\Phi_{\calF}(\Sigma)$ can be identified with
$\Phi_{\calF}(C(\Sigma)\setminus\widetilde C) \times \Phi_{\calF}(\widetilde C)$.
If $\varphi$ is an ${\calF}$--colouring of $\Sigma$, we can associate a
{\em symbol}
$$p^\varphi := \twelvejsym{a_1 & b_1 & c_1}{a_2 & b_2 & c_2}{a_3 & b_3 & c_3}{a_4 & b_4 & c_4}$$
to each $p\in T(\Sigma)$, where the $a_*$'s, the $b_*$'s and the $c_*$'s are shown in Fig.~\ref{fig:col_vert}.
\begin{figure}[ht!]
  \psfrag{p}{\small $p$}
  \psfrag{a1}{\small $a_1$}
  \psfrag{a2}{\small $a_2$}
  \psfrag{a3}{\small $a_3$}
  \psfrag{a4}{\small $a_4$}
  \psfrag{b1}{\small $b_1$}
  \psfrag{b2}{\small $b_2$}
  \psfrag{b3}{\small $b_3$}
  \psfrag{b4}{\small $b_4$}
  \psfrag{c1}{\small $c_1$}
  \psfrag{c2}{\small $c_2$}
  \psfrag{c3}{\small $c_3$}
  \psfrag{c4}{\small $c_4$}
  \centerline{\includegraphics{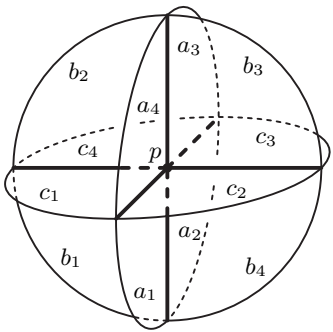}}
  \caption{Colours near a triple point $p$.}
  \label{fig:col_vert}
\end{figure}
Since this definition involves some choices about the identification
of the neighbourhood of $p$ with the abstract picture above, we assume
that each symbol is invariant under changing of this identification.
More precisely, we assume that the identities
$$
\twelvejsym{a_1 & b_1 & c_1}{a_2 & b_2 & c_2}{a_3 & b_3 & c_3}{a_4 & b_4 & c_4}
=
\twelvejsym{c_1 & a_1 & b_1}{c_2 & a_2 & b_2}{c_3 & a_3 & b_3}{c_4 & a_4 & b_4}
=
\twelvejsym{a_4 & c_1 & b_2}{a_1 & c_4 & b_3}{a_2 & c_3 & b_4}{a_3 & c_2 & b_1}
=
\twelvejsym{a_1 & b_4 & c_2}{a_2 & b_3 & c_1}{a_3 & b_2 & c_4}{a_4 & b_1 & c_3}
$$
hold for all $a_*, b_*, c_* \in\calF$.
These yield all the identities corresponding to changes of the identification, because the whole symmetry group of the triple-point neighbourhood (being a semidirect product $(\mathbb Z/_{2\mathbb Z})^3\rtimes{\cal S}_3$, with $48$ elements) is generated by the following three symmetries:
\begin{itemize}
\item the order-3 rotation sending $a_1$ to $b_1$, $b_1$ to $c_1$, and $c_1$ to $a_1$;
\item the order-4 rotation around an horizontal axis, sending $a_1$ to $a_2$;
\item the reflection in the plane containing the $a_*$'s.
\end{itemize}
Note that the triple-point neighbourhood is not completely symmetric; for instance, no symmetry can interchange the region germ coloured by $a_1$ with that coloured by $a_2$, fixing those coloured by $a_3$ and $a_4$.
To be precise, we should distinguish a symbol from its equivalence class; nevertheless,
for the sake of simplicity, we do not make the notation heavier and we use the same notation for both the symbol and its equivalence class.

Let us consider now the polynomial ring $\calR :=\field[s_1,\ldots,s_N]$, where $\field$ is a field
and the $s_*$'s are (the equivalence classes of) the symbols.
The polynomial
$$
SS_m(\Sigma) := \sum_{\varphi\in\Phi_{\calF}(\Sigma)}
\Bigg(\prod_{p\in T(\Sigma)} p^\varphi\Bigg)
$$
of $\calR$ is called {\em state sum of $\Sigma$ of type $m$}.
Note that $SS_m(\Sigma)$ is an invariant of $\Sigma$, but it is not an
invariant of $M$; in fact, it depends on the particular nullhomotopic filling Dehn sphere $\Sigma$ of $M$.

Consider for instance an \saddle--move between $\Sigma$ and $\Sigma'$ (see
Fig.~\ref{fig:s}).
Consider also a colouring of $\Sigma$ and a colouring of $\Sigma'$ matching each other out of the
portions involved in the move; see Fig.~\ref{fig:saddle_col} for the notation.
\begin{figure}[ht!]
  \psfrag{p1}{\small $p_1$}
  \psfrag{p2}{\small $p_2$}
  \psfrag{p3}{\small $p_3$}
  \psfrag{p1p}{\small $p_{1'}$}
  \psfrag{p2p}{\small $p_{2'}$}
  \psfrag{p3p}{\small $p_{3'}$}
  \psfrag{alpha1}{\small $\alpha_1$}
  \psfrag{alpha2}{\small $\alpha_2$}
  \psfrag{alpha3}{\small $\alpha_3$}
  \psfrag{alpha4}{\small $\alpha_4$}
  \psfrag{alpha5}{\small $\alpha_5$}
  \psfrag{alpha6}{\small $\alpha_6$}
  \psfrag{A1}{\small $A_1$}
  \psfrag{A2}{\small $A_2$}
  \psfrag{A3}{\small $A_3$}
  \psfrag{A4}{\small $A_4$}
  \psfrag{A5}{\small $A_5$}
  \psfrag{A6}{\small $A_6$}
  \psfrag{a1}{\small $a_1$}
  \psfrag{a2}{\small $a_2$}
  \psfrag{a3}{\small $a_3$}
  \psfrag{a4}{\small $a_4$}
  \psfrag{a5}{\small $a_5$}
  \psfrag{a6}{\small $a_6$}
  \psfrag{X}{\small $X$}
  \psfrag{Y}{\small $Y$}
  \psfrag{x}{\small $x$}
  \psfrag{y}{\small $y$}
  \centerline{\includegraphics{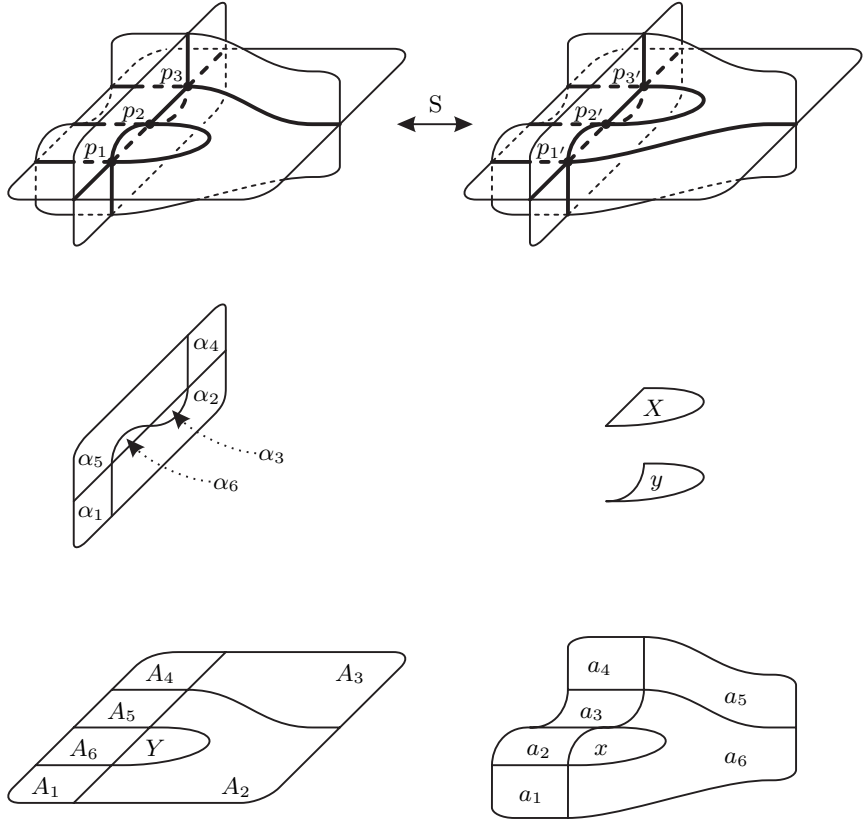}}
  \caption{An \saddle--move between two nullhomotopic filling Dehn spheres with colourings (matching each other
    out of the portions involved in the move).}
  \label{fig:saddle_col}
\end{figure}
We have
\begin{eqnarray*}
SS_m(\Sigma) & = &
\sum_{\varphi\in\Phi_{\calF}(\Sigma)}
\Bigg(\prod_{p\in T(\Sigma)} p^\varphi\Bigg) =
\sum_{\varphi\in\Phi_{\calF}(\Sigma)}
\Bigg(\prod_{\doubleidx{p\in T(\Sigma)\setminus}{\{p_1,p_2,p_3\}}}
p^\varphi\Bigg) (p_1^\varphi p_2^\varphi p_3^\varphi) = \\
& = & \sum_{\doubleidx{\varphi_1\in\Phi_{\calF}(\Sigma\setminus}{\{C_x,C_Y\})}}
\Bigg(\prod_{\doubleidx{p\in T(\Sigma)\setminus}{\{p_1,p_2,p_3\}}}
p^{\varphi_1}\Bigg)
\Bigg(\sum_{\doubleidx{\varphi_2\in\Phi_{\calF}(\{C_x,C_Y\})}{}}
p_1^{\overline{\varphi}} p_2^{\overline{\varphi}} p_3^{\overline{\varphi}}\Bigg),
\end{eqnarray*}
where $\overline{\varphi}$ is the ${\calF}$--colouring given by $\varphi_1$ and
$\varphi_2$; in the last equality we have applied the distributive 	property and the fact that $\varphi_1=\overline{\varphi}$ near every $p\neq p_i\ \forall i=1,2,3$.
Analogously, we have
$$
SS_m(\Sigma') =
\sum_{\doubleidx{\varphi'_1\in\Phi_{\calF}(\Sigma'\setminus}{\{C_X,C_y\})}}
\Bigg(\prod_{\doubleidx{p\in T(\Sigma')\setminus}{\{p_{1'},p_{2'},p_{3'}\}}}
p^{\varphi'_1}\Bigg)
\Bigg(\sum_{\doubleidx{\varphi'_2\in\Phi_{\calF}(\{C_X,C_y\})}{}}
p_{1'}^{\overline{\varphi}'} p_{2'}^{\overline{\varphi}'} p_{3'}^{\overline{\varphi}'}\Bigg),
$$
where $\overline{\varphi}'$ is the $\calF$--colouring given by $\varphi'_1$ and
$\varphi'_2$.
Since $\Sigma$ and $\Sigma'$ coincide out of the portion involved in the move, we have, with a slight abuse of notation,
\begin{eqnarray*}
SS_m(\Sigma) - SS_m(\Sigma') =
\sum_{\doubleidx{\varphi_1\in\Phi_{\calF}(\Sigma\setminus}{\{C_x,C_Y\})}}
\Bigg(\prod_{\doubleidx{p\in T(\Sigma)\setminus}{\{p_1,p_2,p_3\}}}
p^{\varphi_1}\Bigg) \cdot \\
\cdot
\Bigg(
\sum_{\doubleidx{\varphi_2\in\Phi_{\calF}(\{C_x,C_Y\})}{}}
p_1^{\overline{\varphi}} p_2^{\overline{\varphi}} p_3^{\overline{\varphi}}
-
\sum_{\doubleidx{\varphi'_2\in\Phi_{\calF}(\{C_X,C_y\})}{}}
p_{1'}^{\overline{\varphi}'} p_{2'}^{\overline{\varphi}'} p_{3'}^{\overline{\varphi}'}
\Bigg);
\end{eqnarray*}
therefore, the difference between $SS_m(\Sigma)$ and $SS_m(\Sigma')$ is an element of the ideal
generated by
\begin{eqnarray*}
  \sum_{x, Y \in \calF}
  \twelvejsym{\alpha_1 & a_1 & A_1}{\alpha_2 & a_2 & A_2}{\alpha_6 & x & Y}{\alpha_5 & a_6 & A_6}
  \twelvejsym{\alpha_2 & a_3 & A_6}{\alpha_3 & a_2 & Y}{\alpha_5 & x & A_2}{\alpha_6 & a_6 & A_5}
  \twelvejsym{\alpha_3 & a_3 & A_5}{\alpha_2 & a_4 & A_2}{\alpha_4 & a_5 & A_3}{\alpha_5 & a_6 & A_4}
  -
  \sum_{X, y \in \calF}
  \twelvejsym{\alpha_1 & a_1 & A_1}{\alpha_2 & a_2 & A_2}{\alpha_6 & a_5 & A_3}{\alpha_5 & a_6 & A_6}
  \twelvejsym{\alpha_2 & a_3 & A_6}{\alpha_3 & a_2 & A_3}{\alpha_5 & a_5 & X}{\alpha_6 & y & A_5}
  \twelvejsym{\alpha_3 & a_3 & A_5}{\alpha_2 & a_4 & X}{\alpha_4 & a_5 & A_3}{\alpha_5 & y & A_4}
\end{eqnarray*}
for all
$a_1,\ldots,a_6,A_1,\ldots,A_6,\alpha_1,\ldots,\alpha_6 \in \calF$.

Analogously, \ti--~and \bubble--moves give rise to the following generators.

The \tone--move gives rise to
$$
\twelvejsym{\alpha_6 & a_2 & A_1}{\alpha_3 & a_3 & A_2}{\alpha_4 & a_4 & A_3}{\alpha_5 & a_5 & A_4}
\twelvejsym{\alpha_1 & a_1 & B_1}{\alpha_2 & a_2 & B_2}{\alpha_3 & a_5 & B_3}{\alpha_6 & a_6 & B_4}
-$$
$$
\sum_{X_\star, Y_\star, z_\star, \zeta_\star \in \calF}
\twelvejsym{\alpha_1 & a_1 & X_1}{\alpha_2 & z_1 & X_2}{\zeta_2 & z_2 & X_3}{\zeta_1 & a_6 & X_4}
\twelvejsym{\zeta_1 & z_1 & Y_1}{\zeta_2 & a_3 & Y_2}{\alpha_4 & a_4 & Y_3}{\alpha_5 & z_2 & Y_4}
\twelvejsym{\alpha_6 & B_1 & A_1}{\alpha_1 & Y_1 & A_2}{\zeta_1 & Y_2 & X_2}{\alpha_5 & B_2 & X_1}
\twelvejsym{B_2 & a_6 & X_2}{B_3 & z_2 & A_2}{Y_3 & a_4 & A_3}{Y_2 & a_5 & X_3}
\twelvejsym{\alpha_2 & B_4 & X_4}{\alpha_3 & Y_4 & X_3}{\alpha_4 & Y_3 & A_3}{\zeta_2 & B_3 & A_4}
\twelvejsym{B_1 & a_2 & A_1}{B_4 & a_3 & X_1}{Y_4 & z_1 & X_4}{Y_1 & a_1 & A_4}
$$
for all
$a_1,\ldots,a_6,A_1,\ldots,A_4,B_1,\ldots,B_4,\alpha_1,\ldots,\alpha_6 \in \calF$.

The \ttwo--move gives rise to
$$
\sum_{x_\star \in \calF}
\twelvejsym{\alpha_8 & x_1 & A_2}{\alpha_3 & x_2 & A_3}{\alpha_4 & a_4 & A_4}{\alpha_7 & a_5 & A_5}
\twelvejsym{\alpha_1 & a_1 & \beta_8}{\alpha_2 & x_1 & \beta_1}{\alpha_3 & a_5 & \beta_2}{\alpha_8 & a_6 & \beta_3}
\twelvejsym{\beta_8 & a_1 & A_1}{\beta_3 & a_2 & A_2}{\beta_4 & x_2 & A_5}{\beta_7 & x_1 & A_6}
\twelvejsym{\alpha_7 & x_2 & \beta_7}{\alpha_4 & a_2 & \beta_6}{\alpha_5 & a_3 & \beta_5}{\alpha_6 & a_4 & \beta_4}
-$$
$$
\sum_{Y_\star \in \calF}
\twelvejsym{\alpha_1 & a_1 & A_1}{\alpha_2 & a_2 & Y_1}{\alpha_5 & a_3 & Y_2}{\alpha_6 & a_6 & A_6}
\twelvejsym{\alpha_8 & \beta_8 & A_2}{\alpha_1 & \beta_7 & A_3}{\alpha_6 & \beta_6 & Y_1}{\alpha_7 & \beta_1 & A_1}
\twelvejsym{\beta_1 & a_6 & Y_1}{\beta_2 & a_3 & A_3}{\beta_5 & a_4 & A_4}{\beta_6 & a_5 & Y_2}
\twelvejsym{\alpha_2 & \beta_3 & A_6}{\alpha_3 & \beta_4 & Y_2}{\alpha_4 & \beta_5 & A_4}{\alpha_5 & \beta_2 & A_5}
$$
for all
$a_1,\ldots,a_6,A_1,\ldots,A_6,\alpha_1,\ldots,\alpha_8,\beta_1,\ldots,\beta_8 \in \calF$.

The \tthree--move gives rise to
$$
\sum_{x \in \calF}
\twelvejsym{\alpha_6 & x & A_2}{\alpha_3 & a_3 & A_3}{\alpha_4 & a_4 & A_4}{\alpha_5 & a_5 & A_5}
\twelvejsym{\alpha_1 & a_1 & \beta_6}{\alpha_2 & x & \beta_1}{\alpha_3 & a_5 & \beta_2}{\alpha_6 & a_6 & \beta_3}
\twelvejsym{\beta_6 & a_1 & A_1}{\beta_3 & a_2 & A_2}{\beta_4 & a_3 & A_5}{\beta_5 & x & A_6}
-$$
$$
\sum_{Y_\star, v, \mu_\star, \zeta_\star \in \calF}
\twelvejsym{\alpha_1 & a_1 & A_1}{\alpha_2 & a_2 & Y_1}{\zeta_2 & v & Y_2}{\zeta_1 & a_6 & A_6}
\twelvejsym{\zeta_1 & a_2 & \beta_5}{\zeta_2 & a_3 & \mu_1}{\alpha_4 & a_4 & \mu_2}{\alpha_5 & v & \beta_4}
\twelvejsym{\alpha_6 & \beta_6 & A_2}{\alpha_1 & \beta_5 & A_3}{\zeta_1 & \mu_1 & Y_1}{\alpha_5 & \beta_1 & A_1}
\twelvejsym{\beta_1 & a_6 & Y_1}{\beta_2 & v & A_3}{\mu_2 & a_4 & A_4}{\mu_1 & a_5 & Y_2}
\twelvejsym{\alpha_2 & \beta_3 & A_6}{\alpha_3 & \beta_4 & Y_2}{\alpha_4 & \mu_2 & A_4}{\zeta_2 & \beta_2 & A_5}
$$
for all
$a_1,\ldots,a_6,A_1,\ldots,A_6,\alpha_1,\ldots,\alpha_6,\beta_1,\ldots,\beta_6 \in \calF$.

The \tfour--move gives rise to
$$
\sum_{x, Y, \mu, \zeta \in \calF}
\twelvejsym{\alpha_6 & x & A_2}{\zeta & a_3 & A_3}{\alpha_4 & a_4 & A_4}{\alpha_5 & a_5 & Y}
\twelvejsym{\alpha_1 & a_1 & \beta_6}{\alpha_2 & x & \beta_1}{\zeta & a_5 & \beta_2}{\alpha_6 & a_6 & \mu}
\twelvejsym{\beta_6 & a_1 & A_1}{\mu & a_2 & A_2}{\beta_4 & a_3 & Y}{\beta_5 & x & A_6}
\twelvejsym{\zeta & \mu & Y}{\alpha_2 & \beta_4 & A_4}{\alpha_3 & \beta_3 & A_5}{\alpha_4 & \beta_2 & A_6}
-$$
$$
\sum_{x', Y', \mu', \zeta' \in \calF}
\twelvejsym{\alpha_1 & a_1 & A_1}{\alpha_2 & a_2 & Y'}{\alpha_3 & x' & A_5}{\zeta' & a_6 & A_6}
\twelvejsym{\zeta' & a_2 & \beta_5}{\alpha_3 & a_3 & \mu'}{\alpha_4 & a_4 & \beta_3}{\alpha_5 & x' & \beta_4}
\twelvejsym{\alpha_6 & \beta_6 & A_2}{\alpha_1 & \beta_5 & A_3}{\zeta' & \mu' & Y'}{\alpha_5 & \beta_1 & A_1}
\twelvejsym{\beta_1 & a_6 & Y'}{\beta_2 & x' & A_3}{\beta_3 & a_4 & A_4}{\mu' & a_5 & A_5}
$$
for all
$a_1,\ldots,a_6,A_1,\ldots,A_6,\alpha_1,\ldots,\alpha_6,\beta_1,\ldots,\beta_6 \in \calF$.

The \bubble--move gives rise to
$$
\sum_{v_\star, x_\star, Y_\star, \mu_\star, \zeta_\star \in \calF}
\twelvejsym{\zeta_1 & a_1 & A_1}{\zeta_2 & a_2 & Y_1}{\zeta_3 & x_2 & Y_6}{\zeta_4 & x_1 & A_2}
\twelvejsym{\zeta_1 & x_1 & Y_1}{\mu_1 & x_2 & A_1}{\mu_2 & x_3 & Y_5}{\zeta_4 & x_4 & Y_6}
\twelvejsym{\zeta_1 & x_4 & A_1}{\mu_1 & x_3 & Y_2}{\mu_2 & x_2 & Y_4}{\zeta_4 & x_1 & Y_5}
\twelvejsym{\zeta_1 & x_1 & Y_2}{\zeta_2 & x_2 & A_1}{\zeta_3 & a_2 & Y_3}{\zeta_4 & a_1 & Y_4}
\twelvejsym{\mu_1 & v_1 & Y_5}{a_1 & v_2 & Y_4}{a_2 & \zeta_3 & Y_3}{\mu_2 & \zeta_2 & A_1}
\twelvejsym{\mu_1 & \zeta_2 & Y_6}{a_1 & \zeta_3 & Y_5}{a_2 & v_2 & A_1}{\mu_2 & v_1 & A_2}
$$
for all
$a_1,a_2,A_1,A_2 \in \calF$.

Finally, let $I_m$ be the ideal of $\calR$ generated by the polynomials (just listed)
deduced from the \ti--, \bubble-- and~\saddle--moves.
We are now in a position to define the invariant and to prove that it depends, indeed, only on $M$.

\begin{teo}
  The coset
  $$
  inv_m(M) = SS_m(\Sigma) + I_m \in\calR /I_m
  $$
  does not depend on the particular nullhomotopic filling Dehn sphere $\Sigma$ presenting the closed $3$--manifold $M$, and thus it is an invariant of $M$.
\end{teo}

\begin{proof}
  We need to prove that, if $\Sigma_1$ and $\Sigma_2$ are nullhomotopic filling
  Dehn spheres of $M$, then $inv_m(\Sigma_1) = inv_m(\Sigma_2)$.
  By virtue of Theorem~\ref{teo:calculus}, we have that $\Sigma_1$ and $\Sigma_2$ can be
  obtained from each other via a sequence of \ti--, \bubble-- and
  \saddle--moves; hence, the difference between $SS_m(\Sigma_1)$ and $SS_m(\Sigma_2)$
  is an element of the ideal $I_m$.
\end{proof}

\paragraph{Other invariants}
From the invariant $inv_m$, a number of other invariants can be constructed.
(See~\cite{King} for the Turaev--Viro version of these modifications.)
Here we list the ideas behind four of them.
\begin{itemize}
\item[1.]
{\em Colouring edges}

In the definition of the invariant $inv_m$, we have taken into account colourings of the regions of a nullhomotopic filling Dehn sphere $\Sigma$.
The idea is to colour also the edges of $\Sigma$.
Hence, let $\calF'$ be a set of $m'$ colours for the edges.
We can generalise the notion of $\calF$--colouring to that of {\em $(\calF,\calF')$--colouring} and we can consider {\em symbols} taking into account also the colours of the edges; therefore, we get symbols with 18 parameters.
As above, we can consider a state sum $SS_{m,m'}(\Sigma)$ and we can study its modifications under the moves of the calculus; finally, we can get an invariant $inv_{m,m'}$.

\item[2.]
{\em Colour weights}

Let $w$ be a map from $\calF$ to $\field$.
Then we can define the state sum as
$$
SS_m(\Sigma) := \sum_{\varphi\in\Phi_{\calF}(\Sigma)}
\Bigg(\prod_{C\in C(\Sigma)} w(\varphi(C))\Bigg)
\Bigg(\prod_{p\in T(\Sigma)} p^\varphi\Bigg)
.
$$
As usual, we should study its modifications under the moves of the calculus, finally getting an invariant $inv_m^w$.

\item[3.]
{\em Simplifying assumptions}

The explicit computation of the invariant $inv_m(M)$ is quite long.
One way to overcome this problem is to assume some additional identities hold. For instance, we can assume some symbols are zero or we can suppose some colourings are forbidden.

\item[4.]
{\em Radical}

The invariant $inv_m$ can be turned into a (maybe) weaker one by considering the coset (represented by the state sum) with respect to any ideal $I$ containing $I_m$.
The more natural one is the {\em radical} $\sqrt{I_m}$ of $I_m$ ({\em i.e.}~the ideal made up of all the polynomials $P\in\calR$ such that $P^n\in I_m$ for some $n\in\mathbb{N}$).

\end{itemize}

\paragraph{Computation of the invariant}
After the definition of the invariant $inv_m(M)$, the issue of computing it naturally arises.
We have not made any computation as yet, but we describe here two different techniques we plan to use to compute $inv_m(M)$.
\begin{itemize}
\item[1.]
{\em Numerical invariants}

Suppose we have an element of the zero variety associated to $I_m$.
Then we can evaluate the state sum $SS_m(\Sigma)$ at it, getting an element of $\field$ which is obviously an invariant of $M$.
Even if the computation is very easy, this technique has the drawback of finding such an element (and this is, in general, a difficult matter).

\item[2.]
{\em Gr\"obner bases}

Suppose we have a Gr\"obner basis of $\calR /I_m$.
(For an introduction to this subject, we refer the reader to~\cite{Eisenbud}, among many other possible sources.)
Then, we can find the normal form of $inv_m(M)$ and hence we can check whether two closed 3--manifolds share the same invariant or not.

\end{itemize}

\paragraph{Relationship with the Turaev--Viro invariant}
The framework we have used to define the invariant $inv_m$ is analogous to that used to define the Turaev--Viro invariant~\cite{Turaev-Viro}.
However, the two calculuses (and their proofs) used in the definition of the invariants are different.
Hence, the following question naturally arises.
\begin{question}
Are the invariant $inv_m$ and the Turaev--Viro invariant related (in some sense) to each other?
\end{question}

\subsection{Lower bounds for the Matveev complexity}

We conclude with a potential application of the invariant $inv_m$.

Throughout this section we will consider only \ptwoirred\ closed 3--man\-i\-folds.
A 3--manifold is {\em \ptwoirred} if every sphere embedded in it bounds a ball and every projective plane embedded in it (if any) is one-sided.

The Matveev complexity is defined using spines~\cite{Matveev:compl}.
However, as shown in~\cite{Matveev:book}, if the 3--manifold $M$ is \ptwoirred\ and closed, the {\em Matveev complexity} $c(M)$ can be defined
\begin{itemize}
\item zero, if $M$ is the 3--sphere $S^3$, the projective space $\RPthree$ or the lens space $L(3,1)$,
\item the minimal number of tetrahedra among all one-vertex triangulations of $M$, otherwise.
\end{itemize}

It is quite easy to find good estimates for the Matveev complexity, but an exact calculation of it is very difficult.
For instance, precise upper bounds can be easily found by exhibiting suitable triangulations, but lower bounds are usually rough.
Here ``precise'' means that the Matveev complexity is {\em a posteriori} usually very close (if even not equal) to the upper bound.

By using nullhomotopic filling Dehn spheres, a first lower bound for the Matveev complexity can be easily computed.
Let us denote by $cs(M)$ the minimal number of triple points among all nullhomotopic filling Dehn spheres of $M$ (in~\cite{Vigara:calculus} the invariant $cs(M)$ is called {\em nullhomotopic genus~0 triple point number}).
\begin{prop}\label{prop:lower_c_cs}
  Let $M$ be a \ptwoirred\ closed $3$--manifold different from $S^3$, $\RPthree$ and $L(3,1)$.
  Then the inequality
  $$
  c(M)\geqslant\frac{cs(M)}{4}
  $$
  holds.
\end{prop}
\begin{proof}
  Let $\calT$ be a triangulation of $M$ with $c(M)$ tetrahedra; such a triangulation exists, because $M$ is \ptwoirred, closed, and different from $S^3$, $\RPthree$ and $L(3,1)$.
  By Proposition~\ref{prop:tria_to_surf}, we have that $M$ has a nullhomotopic filling Dehn sphere with $4c(M)$ triple points.
  Obviously, we have $cs(M)\leqslant 4c(M)$ and hence we get the thesis.
\end{proof}

In practice, this result seems to be useless in order to find a lower bound for the Matveev complexity directly.
In fact, in order to find an inequality like $c(M)\geqslant n$, we should find another inequality like $cs(M)\geqslant 4n$, and this seems to be at least as difficult as finding the former one.
We describe a potential application of the invariant $inv_m$ to overcome this problem.
If $P\in\calR$ is a polynomial, let $\deg(P)$ be its total degree.
Moreover, for any subset $\calP\subset\calR$, let $\deg(\calP)=\min\{\deg(P) : P\in\calP\}$.
\begin{teo}\label{teo:lower_cs_inv}
  Let $M$ be a \ptwoirred\ closed $3$--manifold different from $S^3$, $\RPthree$ and $L(3,1)$.
  Then the inequality
  $$
  cs(M)\geqslant\deg(inv_m(M))
  $$
  holds for all $m>1$.
\end{teo}
\begin{proof}
  Let $\Sigma$ be a nullhomotopic filling Dehn sphere of $M$ with $cs(M)$ triple points.
  Let $SS_m(\Sigma)\in\calR$ be the state sum of $\Sigma$ and let $inv_m(M) = SS_m(\Sigma) + I_m \in\calR /I_m$ be the invariant coset.
  The set $inv_m(M)$ is, in particular, a subset of $\calR$, hence we can define $\deg(inv_m(M))$; obviously, we have $\deg(inv_m(M))\leqslant\deg(SS_m(\Sigma))$.
  Since $\deg(SS_m(\Sigma))$ is just the number $cs(M)$ of triple points of $\Sigma$, we have $\deg(inv_m(M))\leqslant cs(M)$.
\end{proof}

An obvious application of Proposition~\ref{prop:lower_c_cs} and Theorem~\ref{teo:lower_cs_inv} yields the desired lower bounds.
\begin{cor}\label{cor:bounds}
  Let $M$ be a \ptwoirred\ closed $3$--manifold different from $S^3$, $\RPthree$ and $L(3,1)$.
  Then the inequality
  $$
  c(M)\geqslant\frac{\deg(inv_m(M))}{4}
  $$
  holds for all $m>1$.
\end{cor}

It is worth noting that such a framework can be equally applied to the Turaev--Viro setting, leading to non-sharp lower bounds on complexity, as shown by King~\cite{King}.
Hence, the following question naturally arises.
\begin{question}
Are the lower bounds of Corollary~\ref{cor:bounds} sharp, at least for some closed $3$--manifolds?
\end{question}

\subsection*{Acknowledgements}

I would like to thank Prof.~Sergei Matveev and Simon King for the useful
discussions I have had in the beautiful period I have spent at the Department
of Mathematics in Darmstadt.
I would also like to thank the Galileo Galilei Doctoral School of Pisa and the DAAD
(Deutscher Akademischer Austausch Dienst) for giving me the opportunity to
stay in Darmstadt, and Prof.~Alexander Martin for his willingness.

This paper is dedicated to Paolo.

\begin{small}

\end{small}


\begin{thebibliography}{99}

\bibitem{Aitchison-Matsumotoi-Rubinstein}
\textsc{I.~R.~Aitchison -- S.~Matsumotoi -- J.~H.~Rubinstein},
\textit{Immersed surfaces in cubed manifolds},
Asian J. Math. {\bf 1} (1997), no.~1, 85--95.

\bibitem{Babson-Chan}
\textsc{E.~K.~Babson -- C.~S.~Chan},
\textit{Counting faces of cubical spheres modulo two},
Discrete Math. {\bf 212} (2000), no.~3, 169--183.

\bibitem{Biedenharn-Louck}
\textsc{L.~C.~Biedenharn -- J.~D.~Louck},
``Angular momentum in quantum physics,''
Encyclopedia of Mathematics and its Applications, 8.
Addison-Wesley Publishing Co., Reading, Mass., 1981.
xxxii+716~pp.

\bibitem{Dolbilin-Shtanko-Shtogrin}
\textsc{N.~P.~Dolbilin -- M.~A.~Shtan$'$ko -- M.~I.~Shtogrin},
\textit{Cubic manifolds in lattices},
Izv. Ross. Akad. Nauk Ser. Mat. {\bf 58} (1994), no.~2, 93--107;
translation in 
Russian Acad. Sci. Izv. Math. {\bf 44} (1995), no.~2, 301--313.

\bibitem{Eisenbud}
\textsc{D.~Eisenbud},
``Commutative algebra. With a view toward algebraic geometry,''
Graduate Texts in Mathematics, 150. Springer-Verlag, New York, 1995. xvi+785 pp.

\bibitem{Funar}
\textsc{L.~Funar},
\textit{Cubulations, immersions, mappability and a problem of Habegger},
Ann. Sci. \'Ecole Norm. Sup. (4) {\bf 32} (1999), no.~5, 681--700.

\bibitem{Hass-Hughes}
\textsc{J.~Hass -- J.~Hughes},
\textit{Immersions of surfaces in $3$-manifolds},
Topology {\bf 24} (1985), no.~1, 97--112.

\bibitem{Homma-NagaseI}
\textsc{T.~Homma -- T.~Nagase},
\textit{On elementary deformations of maps of surfaces into $3$-manifolds. I},
Yokohama Math. J. {\bf 33} (1985), no.~1-2, 103--119.

\bibitem{Homma-NagaseII}
\textsc{T.~Homma -- T.~Nagase},
\textit{On elementary deformations of maps of surfaces into $3$-manifolds. II}.
Topology and computer science (Atami, 1986), 1--20, Kinokuniya, Tokyo, 1987.

\bibitem{King}
\textsc{S.~A.~King},
\textit{Ideal Turaev-Viro invariants},
Topology Appl. {\bf 154} (2007), no. 6, 1141--1156.

\bibitem{Matveev:calculus}
\textsc{S.~V.~Matveev},
\textit{Transformations of special spines, and the Zeeman conjecture},
Izv. Akad. Nauk SSSR Ser. Mat. {\bf 51} (1987), no.~5, 1104--1116, 1119;
translation in
Math. USSR-Izv. {\bf 31} (1988), no.~2, 423--434.

\bibitem{Matveev:compl}
\textsc{S.~V.~Matveev},
\textit{The theory of the complexity of three-dimensional manifolds},
Akad. Nauk Ukrain. SSR Inst. Mat. Preprint (1988), no.~13, 32~pp.

\bibitem{Matveev:book}
\textsc{S.~V.~Matveev},
``Algorithmic topology and classification of 3-manifolds,''
Algorithms and Computation in Mathematics, 9. Springer-Verlag, Berlin, 2003.
xii+478 pp.

\bibitem{Montesinos}
\textsc{J.~M.~Montesinos-Amilibia},
\textit{Representing $3$-manifolds by Dehn spheres}.
Mathematical contributions: volume in honor of Professor Joaqu\'\i n Arregui Fern\'andez, 239--247, Homen. Univ. Complut., Editorial Complutense, Madrid, 2000.

\bibitem{Papa}
\textsc{C.~D.~Papakyriakopoulos},
\textit{On Dehn's lemma and the asphericity of knots},
Ann. of Math. (2) {\bf 66} (1957), 1-26.

\bibitem{Piergallini:calculus}
\textsc{R.~Piergallini},
\textit{Standard moves for standard polyhedra and spines},
Rend. Circ. Mat. Palermo (2) Suppl. {\bf 18} (1988), 391--414.

\bibitem{Roseman}
\textsc{D.~Roseman},
\textit{Reidemeister-type moves for surfaces in four-dimensional space}.
Knot theory (Warsaw, 1995), 347--380,
Banach Center Publ., 42, Polish Acad. Sci., Warsaw, 1998.

\bibitem{Shtanko-Shtogrin}
\textsc{M.~A.~Shtan$'$ko -- M.~I.~Shtogrin},
\textit{Embedding cubic manifolds and complexes into a cubic lattice},
Uspekhi Mat. Nauk {\bf 47} (1992), no.~1(283), 219--220;
translation in 
Russian Math. Surveys {\bf 47} (1992), no.~1, 267--268.

\bibitem{Turaev}
\textsc{V.~G.~Turaev},
``Quantum invariants of knots and $3$-manifolds,''
de Gruyter Studies in Mathematics, 18. Walter de Gruyter \& Co., Berlin, 1994. x+588~pp.

\bibitem{Turaev-Viro}
\textsc{V.~G.~Turaev -- O.~Ya.~Viro},
\textit{State sum invariants of $3$-manifolds and quantum $6j$-symbols},
Topology {\bf 31} (1992), 865-902.

\bibitem{Vigara:present}
\textsc{R.~Vigara},
\textit{A new proof of a theorem of J.~M.~Montesinos},
J. Math. Sci. Univ. Tokyo {\bf 11} (2004), no. 3, 325--351.

\bibitem{Vigara:calculus}
\textsc{R.~Vigara},
\textit{A set of moves for Johansson representation of $3$-manifolds},
Fund. Math. {\bf 190} (2006), 245--288. 

\bibitem{Vigara:tesi}
\textsc{R.~Vigara},
``Representaci\'on de 3-variedades por esferas de Dehn rellenantes,''
PhD Thesis.
UNED, Madrid, 2006.

\end{thebibliography}
\end{document}